\documentclass[12pt]{amsart}
\usepackage[utf8]{inputenc}
\usepackage{setspace}
\usepackage[english]{babel}
\usepackage{mathabx}
\usepackage{framed}
\usepackage{soul}
\usepackage[all]{xy}
\usepackage[margin=1in]{geometry}
\usepackage{arydshln}
\usepackage[colorinlistoftodos]{todonotes}
\usepackage{pb-diagram}
\usepackage[mathscr]{euscript}
\usepackage{multicol}
\usepackage{stmaryrd} %\mapsfrom
\usepackage{empheq} %for labeling equations individually in cases
\usepackage{tikz-cd}
\tikzset{
  symbol/.style={
    draw=none,
    every to/.append style={
      edge node={node [sloped, allow upside down, auto=false]{$#1$}}}
  }
}

\usepackage{amsmath,amsfonts,amssymb,amsthm,mathrsfs,latexsym,mathtools}
\usepackage{commath}
\usepackage[T1]{fontenc}
\usepackage{hyperref}
\usepackage{caption} %figure
\usepackage{subcaption} %figure
\usepackage{blkarray}%added for matrices with headings
\usepackage{lipsum}

\DeclareMathAlphabet{\mathbbmsl}{U}{bbm}{m}{sl}

\usetikzlibrary{matrix}

\newcommand{\C}{\mathbb{C}} %newly added on Aug 30, 2021 
\newcommand{\Z}{\mathbb{Z}}

\newcommand{\bbP}{\mathbb{P}}

\newcommand{\Sym}{\textup{Sym}}
\newcommand{\Hilb}{\textup{Hilb}}
\newcommand{\Bl}{\textup{Bl}}
\newcommand{\depth}{\textup{depth}}

\usepackage{amscd} % Aug 19, 2021

\usepackage{enumitem} %Sep 12, 2021

\newtheorem{theorem}{Theorem}[section]
\newtheorem{definition}[theorem]{Definition}
\newtheorem{proposition}[theorem]{Proposition}
\newtheorem{corollary}[theorem]{Corollary}
\newtheorem{question}[theorem]{Question}
\newtheorem{example}[theorem]{Example}
\newtheorem{lemma}[theorem]{Lemma}

\newtheorem{problem}[theorem]{Problem}
\newtheorem{remark}[theorem]{Remark}

\newtheorem{notation}[theorem]{Notation}

\newtheorem{thm}{Theorem}

\tikzset{commutative diagrams/.cd,
mysymbol/.style={start anchor=center,end anchor=center,draw=none}
}

\title{Hilbert Scheme of a Pair of Skew Lines on Cubic Hypersurfaces}
\author{Yilong Zhang}
\address{Department of Mathematics, Purdue University,
150 N. University Street, West Lafayette, IN 47907, U.S.A.}
\email{zhan4740@purdue.edu}

\date{Apr 18, 2025}
\begin{document}

\maketitle
\begin{abstract}
    We study an irreducible component $H(X)$ of the Hilbert scheme $\Hilb^{2t+2}(X)$ of a smooth cubic hypersurface $X$ containing two disjoint lines. For cubic threefolds, $H(X)$ is always smooth, as shown in \cite{YZ_SkewLines}. We provide a second proof and generalize this result to higher dimensions. Specifically, for cubic hypersurfaces of dimension at least four, we show $H(X)$ is normal, and it is smooth if and only if $X$ lacks certain "higher triple lines." We characterize $H(X)$ using the Hilbert-Chow morphism and describe its singularities when $X$ is special. %As an application, when $X$ is a general cubic fourfold, a natural double cover of $H(X)$ yields a smooth model extending a rational map considered by Voisin. %The key point is to show the reducedness (and smoothness) of certain scheme-theortical intersections.   
\end{abstract}

%\bibliographystyle{alpha}
%\bibliographystyle{plain}
%\bibliography{bibfile}

\section{Introduction}

Lines in the projective space $\mathbb P^n$ are parameterized by the Grassmannian $Gr(2,n+1)$ of two-dimensional linear subspaces in $\C^{n+1}$. Using Grothendieck's language, the Grassmannian $Gr(2,n+1)$ is the Hilbert scheme $\Hilb^{t+1}(\mathbb P^n)$ parameterizing universal family of subschemes of $\mathbb P^n$ with Hilbert polynomial $t+1$, namely projective lines.

One can consider a pair of lines. If $n$ is at least $3$, then a general pair of two lines in $\mathbb P^n$ is skew, and they define an irreducible component $H(\mathbb P^n)$ of the Hilbert scheme $\Hilb^{2t+2}(\mathbb P^n)$. The component $H(\mathbb P^n)$ is birational to the symmetric product $\Sym^2Gr(2,n+1)$ and parameterizes pairs of disjoint lines and their flat limits.

Using deformation theory, Chen, Coskun, and Nollet showed that
%One can refer to \cite{CCN,YZ_SkewLines} for details.

\begin{theorem}\cite{CCN}
    The component $H(\mathbb P^n)$ is smooth.
\end{theorem}

In this paper, we study a similar question for cubic hypersurfaces in $\mathbb P^n$.

\subsection{Cubic hypersurfaces}
The study of lines on cubic hypersurfaces dates back to the 19th century, when Cayley and Salmon discovered the 27 lines on a cubic surface \cite{LuigiCubic}. There are 16 lines that skew to a given line, so there are $27\times 16=432$ disjoint pairs of skew lines on a cubic surface. 

Starting from dimension 3, lines on a cubic hypersurface vary in a continuous family. They form a subvariety $F$ of Grassmannian and are called the \textit{Fano variety of lines}. For example, when $X$ is a cubic threefold, then $F$ is a surface of general type. Griffiths and Clemens \cite{CG} studied the Abel-Jacobi map
\begin{equation}\label{eqn_IntroAJ}
    F\times F\to J(X),
\end{equation}
by integrating a differential 3-form $\omega$ against a topological 3-chain bounding a pair of lines $L_1$ and $L_2$. Its geometry is used to prove the irrationality of cubic threefolds.

For a cubic fourfold, $F$ is hyperk\"{a}hler fourfold \cite{BeaDon85} and deformation equivalent to the Hilbert scheme of two points on a K3 surface. Voisin defined a rational map from a pair of two skew lines on cubic fourfold to a hyperkahler 8-fold considered by Voisin \cite{VoisinMap}. Its restriction to a general hyperplane section is precisely \eqref{eqn_IntroAJ}.  This motivates us to study the parameter space of a pair of skew lines on a cubic hypersurface. 
%\begin{equation}\label{eqn_voisin}
%   F\times F\dashrightarrow Z 
%\end{equation}
 %Since a general pair of lines on $X$ are skew when $\dim(X)\ge 3$, both map \eqref{eqn_IntroAJ} and \eqref{eqn_voisin} are determined by where pairs of skew lines are sent to. 

Let $X\subseteq \bbP^{n}$ be a smooth cubic hypersurface, with $n\ge 4$, then a general pair of two lines on $X$ is disjoint and determines an irreducible component $H(X)$ in the Hilbert scheme $\Hilb^{2t+2}(X)$ of $X$. Then, there is a closed embedding
$$i: H(X)\hookrightarrow H(\mathbb P^n).$$

\begin{definition}\normalfont
    We will call the component $H(X)$ the \textit{Hilbert scheme of a pair of skew lines} on $X$.
\end{definition}

In a previous work, we proved that

\begin{theorem}  \cite{YZ_SkewLines} \label{thm_H(X)3fold}
When $X$ is a smooth cubic threefold, $H(X)$ is smooth.  
\end{theorem}

The proof is based on the Abel-Jacobi map and the geometry of the theta divisor of the intermediate Jacobian. In this paper, we will provide a second proof, and the new argument naturally generalizes to higher dimensions.

\begin{theorem}\label{thm_main}
    Suppose $X$ is a general cubic hypersurface of dimension at least four. The Hilbert scheme of a pair of skew lines $H(X)$ is smooth. 
\end{theorem}

Our method is to describe the birational geometry of $H(X)$. We relate to the Chow variety $ \Sym^2F$ parameterizing pairs of lines on $X$ with reduced structure. There is a Hilbert-Chow morphism which factors through the Hilbert scheme of two points of $F$
\begin{equation}\label{eqn_IntroHC}
    H(X)\xrightarrow{\sigma_2} F^{[2]}\xrightarrow{\sigma_1} \Sym^2F.
\end{equation}

The map $\sigma_1$ blows up the diagonal, and the second map $\sigma_2$ blows up the strict transform of the subvariety $D_F'$ of $\Sym^2F$ parameterizing pairs of incidental lines.

\begin{figure}[ht]
\centering
\includegraphics[width=0.6\textwidth]{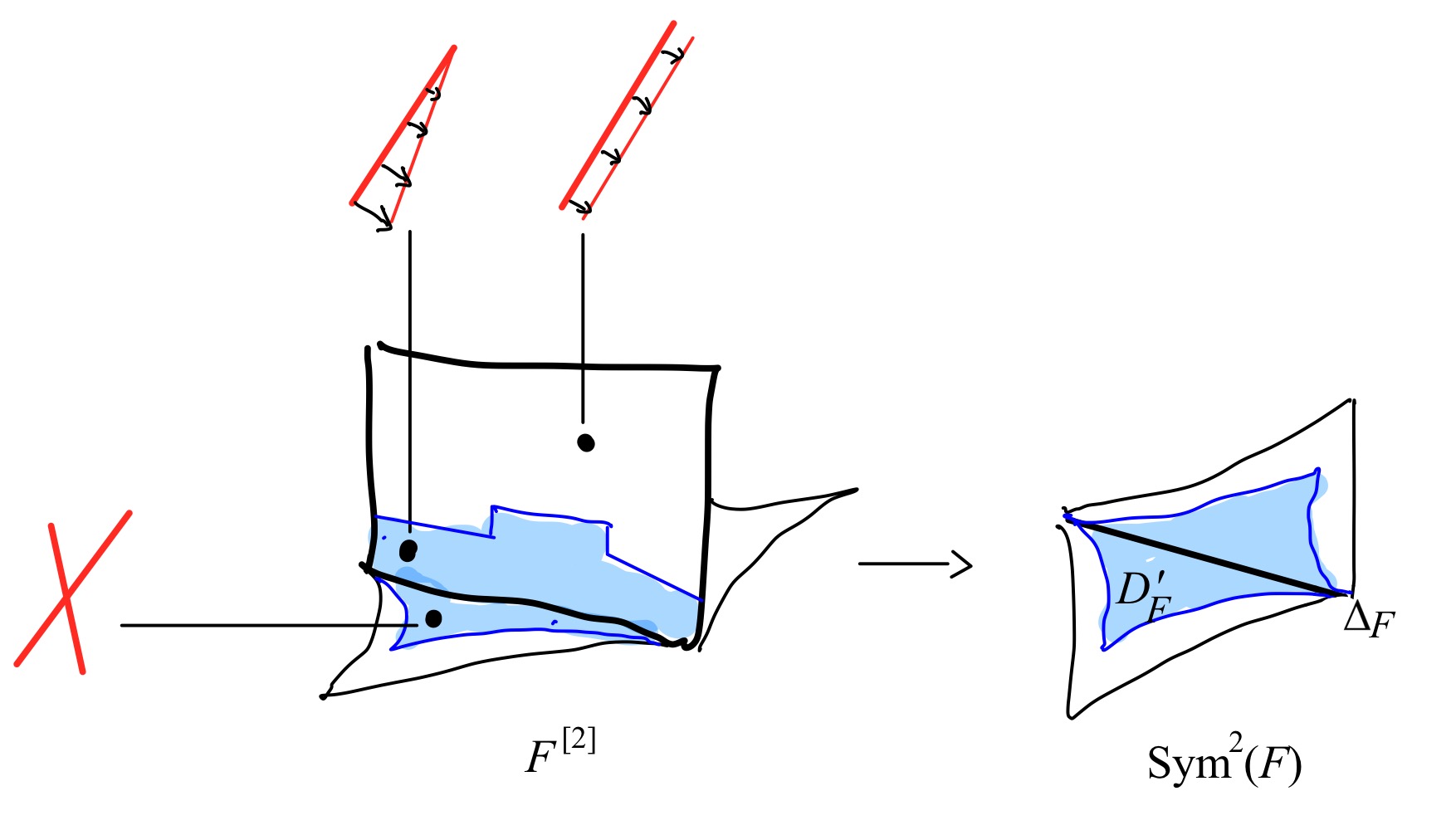}
\caption{\label{figure_intro} $F^{[2]}$ as an "intermediate" Hilbert scheme of a pair of skew lines}
\end{figure}

The composition \eqref{eqn_IntroHC} is an isomorphism at the locus of pairs of disjoint lines, which is called the type (I) subscheme of $X$, as in \cite{CCN}; At a general point of $D_F'$ parameterizing two lines $(L_1,L_2)$ intersecting at a point, the blow-up $\sigma_2$ provides collections of $\mathbb P^3$ containing the plane $\textup{Span}(L_1,L_2)$, hence determining an embedded point supported on the intersection  $L_1\cap L_2=\{x\}$. Such a subscheme is of type (III). Over the diagonal, the exceptional locus $E$ of $\sigma_1$ parameterized a line $L$ together with a normal direction $v\in H^0(N_{L|X})$. Depending on if $v$ comes from a sub-line bundle $\mathcal{O}_L$ or $\mathcal{O}_L(1)$, the infinitesimal data will be different. If $v\in H^0(\mathcal{O}_L)$, then it corresponds to a point on $F^{[2]}$ away from the second blow-up center and determines a subscheme of type (II); while if $v\in H^0(\mathcal{O}_L(1))$, it is on the center of $\sigma_2$, and the second blowup in the same way provides a $\mathbb P^3$ containing the plane $\textup{Span}(L,v)$ and determines an embedded point supported on $L$. This is called a type (IV) subscheme. 

In figure \ref{figure_intro} we depict the first blow-up. The blue region denotes the incidental locus $D_F'$ and its strict transform. The non-equidistribution of normal directions over the diagonal $\Delta_F$ is a reflection of the fact that there is a distinction between two types of lines based on the splitting data of the normal bundle $N_{L|X}$. We also note that the portrait is for $\dim(X)\ge 4$. When $\dim(X)=3$, $D_F'$ does not contain the diagonal.

To understand Theorem \ref{thm_H(X)3fold}, when $\dim(X)=3$, the subvariety $D_F'$ is a divisor, so the second blowup $\sigma_2$ is an isomorphism. Hence, the Hilbert scheme of a pair of skew lines $H(X)$ is always smooth as long as $X$ is smooth. However,  when $\dim(X)\ge 4$, $D_F'$ has higher codimension, so if its strict transform is singular, the second blow-up will produce singularities on $H(X)$. We need to describe when this will happen.

%so the space $\Bl_{\Delta_F}\Sym^2F$ in the middle is smooth. Also note is isomorphic to the Hilbert scheme

%It can be regarded as the zero locus of a section of $\Sym^3(\mathcal{S}^*)$ over Grassmannian $Gr(2,n+1)$, where $\mathcal{S}^*$ is the dual of the universal subbundle. $F$ is smooth as long as the cubic hypersurface $X$ is smooth \cite[Cha. 2 Prop. 1.19]{Huy}.

%When $\dim(X)\ge 3$, $F$ has positive

%  It is smooth as long as the cubic hypersurface is smooth. Two general lines on $X$ are skew.

\subsection{Higher triple lines}
Let $X\subseteq \mathbb P^n$  be a cubic hypersurface. A line $L\subseteq X$ is called of the \textit{first type} if there is a codimension-three linear subspace $P^{n-3}\subseteq \mathbb P^{n}$ containing $L$ and is tangent to $X$ at all points on $L$. It is called a line of the \textit{second type} if there is a codimension-two linear subspace $P^{n-2}\subseteq \mathbb P^{n}$ tangent to $X$ along $L$. Lines of the second type form a closed subspace in $F$. They are distinguished based on the first-order data of the cubic equation near $L$. These notions are introduced and studied in \cite{CG}.

We call a line $L\subseteq X$ on a smooth cubic hypersurface a \textit{higher triple line} if it is the second type, and the second-order data of the cubic equation near $L$ is linearly dependent when restricted to the codimension two space $P^{n-2}$. The three notions are related by the specialization
$$\textup{lines of the first type} \rightsquigarrow \textup{lines of the second type}\rightsquigarrow \textup{higher triple lines}.$$

Alternatively, a line has the second type if there is a linear codimension two subspace $P^{n-2}$ and $P^{n-2}\cap X$ is singular at all points at $L$. When $X$ is futhur a higher triple line, $P^{n-2}\cap X$ has transversal $A_2$ singularity everywhere on the line.
%moreover has $A_2$ singularity in the transversal direction everywhere on the line. 

%Alternatively, take the Taylor expansion of the defining equation of $X$ with respect to the variable defining a given line $L$. The 1st-order data are enough to distinguish between the lines of the first and the second type. This is studied in \cite{CG}. However, to further distinguish the higher triple line, one requires 2nd-order data. We refer to Section \ref{sec_highertriplelines} for a detailed discussion.

For example, when $X$ is a cubic threefold, a higher triple line is a triple line, i.e., there is a plane $\mathbb P^2$ such that $\mathbb P^2\cap X=3L$; When $X$ is a cubic fourfold, being a higher triple line is more restrictive, it implies $L$ is contained in a cone of a cuspidal cubic curve as a $\mathbb P^3$-section of $X$ (see Figure \ref{figure_coneofcusp}). For special cubic fourfolds, higher triple lines may vary in a continuous family — Fermat cubic fourfold has 45 Eckardt points, and corresponds to (at least) 45 one-parameter families of higher triple lines. However, we show that having a higher triple line is a proper closed condition in the moduli space of cubic hypersurfaces. 

\begin{proposition}(cf. Proposition \ref{prop_dim-count}) \label{prop_Intro_general}
    A general cubic hypersurface does not contain a higher triple line.
\end{proposition}

The second main result of this paper is

\begin{theorem}\label{thm_mainprecise}
Let $X$ be a smooth cubic hypersurface with $\dim(X)\ge 4$. Then $H(X)$ is smooth if and only if the $X$ has no higher triple lines.
\end{theorem}

Then, Proposition \ref{prop_Intro_general} and Theorem \ref{thm_mainprecise} recover Theorem \ref{thm_main}.

\subsection{Singularities of $H(X)$}
For special smooth cubic hypersurface $X$ when $\dim(X)\ge 4$, we want to characterize the singularities of $H(X)$. We denote $\widetilde{H(X)}$ the fiber product of $H(X)$ with the double cover $F\times F\to \Sym^2F$.
\begin{proposition} (cf. Corollary \ref{cor_H(X)normal}, Proposition \ref{prop_H(X)tilde-hypersing}) 
    $H(X)$ is normal and has its natural branched double $\widetilde{H(X)}$ has hypersurface singularities.\end{proposition}

In addition, just like a pair of skew lines span a $\mathbb P^3$, each subscheme $Z\in H(X)$ of $X$ determines a unique 3-dimensional projective subspace $\mathbb P^3_Z$ of $\mathbb P^n$.

\begin{proposition} (cf. Proposition \ref{prop_singcubicsurface})
  Suppose $Z\in H(X)$ is a singularity,  then the corresponding $\mathbb P^3_Z$ is either contained in $X$ or intersect $X$ along a cubic surface $C$. In the latter case, the cubic surface $C$ is 
    \begin{itemize}
        \item  a cone of a planer cuspidal cubic curve, or 
        \item the union of a plane and a quadric cone meeting tangentially along a line.
    \end{itemize}
\end{proposition}

\subsection{Strategy of the proof}
%We will prove Theorem \ref{thm_main} through the birational geometry of $H(X)$, or rather the Hilbert-Chow morphism. 
The proof is based on the description of the Hilbert-Chow morphism \eqref{eqn_IntroHC}. In fact, we are oversimplifying the discussion above — To show that the Hilbert-Chow morphism has the factorization \eqref{eqn_IntroHC}, we take into account the scheme structure of the blow-up centers. In other words, we need to know whether the scheme structure of the blow-up center is \textit{reduced}.

%For $\mathbb P^n$, one can factor the Hilbert-Chow morphism through successful blowups (cf. \cite[p.8]{CCN}, \cite[Prop. 3.3]{YZ_SkewLines})
%\begin{equation} \label{eqn_Pntwostepblowup}
%  H(\mathbb P^n)\xrightarrow{f_2} \Bl_{\Delta}\Sym^2Gr(2,n+1)\xrightarrow{f_1} \Sym^2Gr(2,n+1),
%\end{equation}
%where $f_1$ is the blowup of the diagonal of the symmetric square, and the middle term is identified with the Hilbert scheme of two points on $Gr(2,n+1)$. $f_2$ is the blowup of the proper transform of the locus of a pair of two incidental lines.

By the universal property of blow-up, the birational morphism $H(X)\to \Sym^2F$ is the restriction of the Hilbert-Chow morphism of the Hilbert scheme of a pair of skew lines on the ambient projective space. To avoid talking about singularities on the diagonal, we can take a branched double cover $\widetilde{H(X)}$ of $H(X)$ by ordering the pair of lines. Now, the Chow variety becomes $F\times F$ and is smooth.

The Hilbert-Chow morphism on the branched double cover of $H(\mathbb P^n)$ factors as blowup of the diagonal, and then blow up the strict transform $\tilde{D}$ of the subvariety $D$ parameterizing pairs of incidental lines (cf. \cite[Prop. 3.3]{YZ_SkewLines}). The factorization is shown in the second column of the following diagram

%We have the following diagram with second column being the double cover of \eqref{eqn_Pntwostepblowup}.

%Then, the Hilbert-Chow morphism \eqref{eqn_HC} factors through the first column of the following diagram and includes into a double cover of \eqref{eqn_Pntwostepblowup} in the second column.%Therefore the Hilbert-Chow morphism   
%\begin{equation} \label{eqn_HC}
%    \widetilde{H(X)}\to F\times F
%\end{equation}
%factors through $blowup $

%arising from double cover of \eqref{eqn_IntroHC}.this is the two-step blow-up with the centers being the pullback from the ambient projective space 

%More specifically, 

\begin{figure}[ht]
    \centering
\begin{equation}\label{eqn_intro-diagram}
\begin{tikzcd}
 \widetilde{H(X)}\arrow[d,"\sigma_2"]\arrow[r,hookrightarrow]  &  \widetilde{H(\mathbb P^n)}\cong \textup{Bl}_{\tilde{D}}\textup{Bl}_{\Delta}\big (Gr(2,n+1)\times Gr(2,n+1)\big )\arrow[d] &\\
\textup{Bl}_{\Delta_F}(F\times F) \arrow[d,"\sigma_1"] \arrow[r,hookrightarrow]  & \textup{Bl}_{\Delta}\big (Gr(2,n+1)\times Gr(2,n+1)\big )\arrow[d]\arrow[r,hookleftarrow]&\tilde{D}\arrow[d]\\
F\times F\arrow[r,hookrightarrow]& Gr(2,n+1)\times Gr(2,n+1)\arrow[r,hookleftarrow]&D.
\end{tikzcd}
\end{equation}
\end{figure}

In the first column, $\sigma_1$ is the blow-up of the diagonal $\Delta_F$ of $F\times F$ with reduced structure. $\sigma_2$ is to blow up the scheme-theoretic intersection 
\begin{equation}\label{eqn_schemeintsect}
   \tilde{D}_F:=\textup{Bl}_{\Delta_F}(F\times F)\cap \tilde{D}.
\end{equation}

If the the blow-up center $\tilde{D}_F$ is smooth, then $\widetilde{H(X)}$, together with its $\Z_2$ quotient, $H(X)$, will be smooth. So, proving Theorem \ref{thm_H(X)3fold} and \ref{thm_mainprecise} reduces to 

\begin{problem} \label{problem}
   Determine when the scheme-theoretical intersection \eqref{eqn_schemeintsect} is smooth. 
\end{problem}

First, note that the intersection \eqref{eqn_schemeintsect} is \textit{not transverse} in the sense that the codimension of the intersection is not the sum of the codimensions in the ambient space $ \textup{Bl}_{\Delta}\big (Gr(2,n+1)\times Gr(2,n+1)\big )$. In fact, $\textup{codim}(\tilde{D})=n-2$, $\textup{codim}(\textup{Bl}_{\Delta_F}(F\times F))=8$, but the scheme \eqref{eqn_schemeintsect} has codimension $n+5$. 
This does not forbid the intersection to be smooth. As an example, two lines intersecting at a point in $\mathbb P^3$ are not transverse, but their intersection (a reduced point) is still a smooth variety.

To show the scheme \eqref{eqn_schemeintsect} is smooth, it is equivalent to show that the intersection is \textit{clean} (cf. \cite[Sec. 5.1]{Li}), i.e., their set-theoretical intersection is smooth. This amounts to showing that the intersection of the tangent spaces
\begin{equation}\label{eqn_intro-tangentspace}
    T_{\textup{Bl}_{\Delta_F}(F\times F),p}\cap T_{\tilde D,p}
\end{equation} 
has the expected dimension at all points.

We found that this is true for a general cubic hypersurface $X$. However, for certain special $X$, the intersection \eqref{eqn_intro-tangentspace} may have a larger dimension at certain points above the diagonal. Supposedly, the scheme structure at these poins may be non-reduced, but we shall show this will not happen.%We will characterize these "special" cubic hypersurfaces below.

The core of this paper is to solve Problem \ref{problem} by showing

\begin{proposition}\label{intro_keyprop}
Let $X$ be a smooth cubic hypersurface with $\dim(X)\ge 3$. Then, the scheme-theoretical intersection \eqref{eqn_schemeintsect} is reduced and irreducible. Moreover,  the intersection \eqref{eqn_schemeintsect} is smooth if and only if $X$ has no higher triple lines.
\end{proposition}

Our method is rather direct — we compute the rank of the Jacobian matrix explicitly. The computation is based on the analysis of the second-order data of the cubic equation near a given line. The difficulty lies in the locus over the lines of the second type on the diagonal. Computation on the Jacobian matrix excludes the possibility of additional component on the exceptional locus and shows irreducibility. For reducedness, we use Serre's criterion and show the depth of a local ring of an affine chart of \eqref{eqn_schemeintsect} is at least one. We reduce the proof to bound the depth of $A/I$, where $A$ is a regular local ring and $I$ is generated by two elements. The geometry of the higher triple lines on $X$ provides an upper bound of dimensions of singular locus, which controls the height of the prime ideal to localize and contributes to the proof of reducedness. %In addition, the rank of the Jacobian matrix drops at most by 1, so $\dim(T_p\tilde{D}_F)\le \dim(\tilde{D}_F)+1$ at all points, thus $\tilde{D}_F$ has hypersurface singularities.

As a consequence, when $\dim(X)=3$, $\tilde{D}_F$ has codimension one, and thus a Cartier divisor. Blowing up a Cartier divisor gives the isomorphism $\widetilde{H(X)}\cong \Bl_{\Delta_F}(F\times F)$, hence we recover the smoothness of $H(X)$ (cf. Theorem \ref{thm_H(X)3fold}), although $\tilde{D}_F$ is singular when the cubic threefold has a triple line. In $\dim(X)\ge 4$, $\tilde{D}_F$ has higher codimensions, so it contributes to the singularities of $H(X)$ if it is singular. Therefore, Proposition \ref{intro_keyprop} implies Theorem \ref{thm_mainprecise}. The normality of $H(X)$ follows from reducedness of intersection \eqref{eqn_schemeintsect}.

%explain proof of reducedness, delete diagram

\subsection{Relations to other work}
\subsubsection*{Voisin's map}
When $X$ is a smooth cubic fourfold without a plane, there is a hyperk\"ahler variety $Z$ associated with the Hilbert scheme of twisted cubics \cite{LLSvS}. Voisin \cite[Prop. 4.8]{VoisinMap} constructed a dominant rational map 
\begin{equation}\label{eqn_Voisinmap}
  \psi:F\times F\dashrightarrow Z.  
\end{equation}

In \cite{Chen}, the author proposed a resolution by blowing up the incidental subvariety $D_F$. However, the total space is singular.

Given our smoothness result of the Hilbert scheme of a pair of skew lines (cf. Theorem \ref{thm_main}), it is natural to ask
\begin{question}
    Does the rational map extend to $\widetilde{H(X)}$? Is there a modular interpretation?
\end{question}

\subsubsection*{Singularities of $F_2$}
The Fano variety of lines $F$ has a subvariety $F_2$ parameterizing lines of the second type on $X$ and $\dim(F_2)=\frac12\dim(F)$. When $X$ is a cubic threefold, $F_2$ is a curve, and its singularities correspond to triple lines on $X$ \cite{triplelines}. We may ask
\begin{question}\normalfont\label{Question_F2sing}
  When $\dim(X)\ge 4$, what is the relationship between singularities of $F_2$ and higher triple lines? 
\end{question}

\subsubsection*{Noether-Lefschetz locus}
By Proposition \ref{prop_Intro_general}, the set of cubic fourfolds with a higher triple line $\mathcal{C}_{HTL}$ has codimension at least one in the moduli space $\mathcal{C}$ of cubic fourfolds. So we can say such cubic fourfolds are "special". In contrast, there is an infinite sequence of divisors $\mathcal{C}_d$ of $\mathcal{C}$ parameterized by the discriminant $d$ of the algebraic classes in $H^4$ \cite{Has00}.  Nicolas Addington asked the following question:
\begin{question}\normalfont
   Is $\mathcal{C}_{HTL}$ contained in a Hassett’s Noether-Lefschetz divisor $\mathcal{C}_d$?
\end{question}

Note that $\mathcal{C}_{HTL}$ contains the locus $\mathcal{C}_{E}$ consisting of cubic fourfolds with an Eckardt point (cf. Proposition \ref{prop_Eckardt}), and $\mathcal{C}_{E}$ is contained in $\mathcal{C}_8$ since every such cubic fourfold contains a plane \cite{LPZ18}. In the light of the recent work of Katzarkov, Kontsevich, Pantev, and Yu showing a generic cubic fourfold in $\mathcal{C}_8$ is irrational, one may also ask if a generic member in $\mathcal{C}_{HTL}$ is rational.\\

\noindent\textbf{Acknowledgement}. The author thanks Izzet Coskun for asking him for the higher-dimensional version of Theorem \ref{thm_H(X)3fold} a few years ago, which led to the present paper. I'd like to thank many individuals for discussions on various aspects of this project: David Anderson, Herb Clemens, Laure Flapan, Franco Giovenzana, Luca Giovenzana, Lisa Marquand, Fanjun Meng, Takumi Murayama, Wenbo Niu, and Vaibhav Pandey. Special thanks go to Mircea Mustaţă for his assistance with the reducedness argument, particularly for sharing Lemma \ref{lemma_pdtrick}.\\

\noindent\textbf{Structure of the Paper}.  In Section \ref{sec_highertriplelines}, we review results about lines on a cubic hypersurface and introduce the notion of higher triple lines. In Section \ref{sec_incidental}, we review the classical theory on Grassmannian and a particular Schubert subvariety and desingularization. We also study its intersection with the Fano variety of lines for a cubic hypersurface. In Section \ref{sec_DFtilde}, we study first-order data of the defining equations of \eqref{eqn_schemeintsect} in an affine chart and describe the intersection \eqref{eqn_schemeintsect} set-theoretically. In Section \ref{sec_general}, we compute the rank of the Jacobian matrix of the intersection \eqref{eqn_schemeintsect}, prove the irreducibility of \eqref{eqn_schemeintsect}, and derive conditions when the rank is lower than expected. In Section \ref{sec_reducedness}, we prove the reducedness of \eqref{eqn_schemeintsect}. In Section \ref{sec_app}, we prove the main theorems and derive a few consequences, including the smoothness of $H(X)$ for cubic threefolds (cf. Theorem \ref{thm_H(X)3fold}), the normality of $H(X)$ in all dimensions (cf. Theorem \ref{thm_mainprecise}), and characterize the singularities of the intersection \eqref{eqn_schemeintsect} and the incidence variety. In Section \ref{sec_dimcount}, we prove, using correspondence, that a general cubic hypersurface does not have a higher triple line, leading to the proof of Theorem \ref{thm_main}. In Section \ref{sec_singualarH(X)}, we characterize the singularities of $H(X)$ in terms of the subscheme of type (IV).\\

\noindent\textbf{Notations}. 
\begin{itemize}

    \item $X$, cubic hypersurface in $\mathbb P^n$
    \item $F$, Fano variety of lines of $X$
    \item $H(X)$ and $\widetilde{H(X)}$, Hilbert scheme of a pair of skew lines of $X$ and its branched double cover
    \item $Gr(2,n+1)$, Grassmannian of lines in $\mathbb P^n$
    \item $D$ and $\tilde{D}$, incidental subvariety of $Gr(2,n+1)^2$ and its strict transform 
    \item $D_F$, incidence subvariety of $F\times F$
    \item $\Delta_F$, diagonal of $F\times F$
    \item $\tilde{D}_F$, the scheme theoretical intersection $\textup{Bl}_{\Delta_F}(F\times F)\cap \tilde{D}$ in $\Bl_{\Delta}Gr(2,n+1)^2$
    \item  $V_{n}$, the Segre embedding of $\mathbb P^1\times \mathbb P^{n-2}$
\end{itemize}
%\newpage

%\tableofcontents

\section{Lines on Cubic Hypersurfaces}\label{sec_highertriplelines}
In this section, we will first review the basic results on Fano varieties of lines. One can refer to \cite{AK77} and \cite{Huy} for the details. Then we introduce the notion of higher triple lines and discuss a few properties.

Let $X\subseteq \bbP^{n}$ be a smooth cubic hypersurface with $n\ge 4$. Then the \textit{Fano variety of lines}  $$F=\{L\in Gr(2,n+1)|L\subseteq X\}$$ of $X$ is smooth and has dimension $2n-6$. When $\dim(X)=3$, $F$ is a surface of general type, when $\dim(X)=4$, $F$ is a hyperk\"ahler 4-fold, when $\dim(X)\ge 5$, $F$ is Fano, i.e., anticanonical bundle is ample.

\subsection{Lines of the first and second type}

\begin{definition}\normalfont \label{def_1st2ndtype}
\begin{enumerate}
    \item  A line $L\subseteq X$ is of the first type if $N_{L|X}\cong \mathcal{O}\oplus \mathcal{O}\oplus \mathcal{O}(1)\oplus \cdots \oplus\mathcal{O}(1)$.   
    \item  A line  $L\subseteq X$ is of the second type if $N_{L|X}\cong \mathcal{O}(-1)\oplus \mathcal{O}(1)\oplus \mathcal{O}(1)\oplus \cdots \oplus\mathcal{O}(1)$.
\end{enumerate}
\end{definition}

A general line $L\in F$ is of the first type. Equivalently, the dual map along a line of the first and the second type have different degrees, which can tell them apart.

\begin{lemma}\cite[Lem. 6.7]{CG} \cite[Cha. 2, Cor. 2.6]{Huy}\label{lemma_(n-3)plane}
   Let $L\subseteq X$ be a line on a cubic hypersurface. Then (1) $L$ is of the first type iff is a unique $(n-3)$-plane $P^{n-3}$ tangent to $X$ along $L$.  (2) $L$ is of the second type iff is a unique $(n-2)$-plane $P^{n-2}$ tangent to $X$ along $L$. 
\end{lemma}

One may interpret Lemma \ref{lemma_(n-3)plane} as saying $X$ is "flatter" around a line of the second type.

Let $L$ be a line of the second type in $X$ with equation $x_2=\cdots=x_{n}=0$. Then, by change of coordinates, we can express the equation of cubic threefold $X$ as \cite[6.10]{CG}

\begin{equation}\label{eqn_cubicnfold@2ndtype}
   F(x_0,\ldots,x_n)=x_2x_0^2+x_3x_1^2+\sum_{2\le i,j\le n}x_ix_jL_{ij}(x_0,x_1)+C(x_2,\cdots,x_n),
\end{equation}
where $L_{ij}(x_0,x_1)=L_{ji}(x_0,x_1)$ is a linear homogeneous polynomial, and $C(x_2,\ldots,x_n)$ is a homogeneous cubic. Then $P^{n-2}$ in Lemma \ref{lemma_(n-3)plane} is given by $x_2=x_3=0$.

\begin{lemma}\cite[Cor. 7.6]{CG} \cite[Cha. 2, Ex. 2.14]{Huy} \label{lemma_2ndtypedim}
    The space of lines of the second type $F_2\subseteq F$ has dimension $\frac{1}{2}\dim (F)=n-3$.
\end{lemma}

$F_2$ is smooth for general $X$ \cite[Prop. 2.13]{Huy}, but it may be singular for special $X$.

\begin{definition}\label{def_tripleline}
    A line $L$ of $X$ is called a triple line if there is a plane $P^2$ containing $L$, such that either $X\cap P^2=3L$, or $P^2\subseteq X$ \footnote{Here, we allow $P^2$ to be contained in $X$ for convenience. This agrees with the existing definition in \cite{CG} and \cite{GK} because neither a smooth cubic threefold nor a general cubic fourfold contains a plane.}.
\end{definition}

For cubic threefolds, $F_2$ is a curve in $F$, and triple lines are precisely the singularities of $F_2$ \cite{triplelines}. In particular, a triple line is of the second type, and there are at most finitely many triple lines \cite[Lem. 10.15]{CG}. In addition, a general cubic threefold does not have a triple line. %In addition, it is shown in \cite{triplelines} that triple lines are precisely the singularities of the curve $F_2$. 

However, the properties of the triple line are not parallel in higher dimensions: for cubic fourfold, a triple line could be of the first type and varies in a two-parameter family, even for the general ones \cite{GK}. We instead study a variant notion whose properties are shared for cubic hypersurfaces in all dimensions. It naturally appears when finding the singularities of the Hilbert scheme of a pair of skew lines $H(X)$, which will be explored in Section \ref{sec_general}.

\subsection{Higher triple lines}

\begin{definition}\normalfont\label{def_htL}
    We call a line $L$ on a smooth cubic hypersurface $X$ a \textit{higher triple line} if 
    \begin{itemize}
        \item $L$ is of the second type, and
        \item in terms of the coordinates \eqref{eqn_cubicnfold@2ndtype}, the columns of the matrix of linear forms \begin{equation}
S:=\begin{bmatrix}\label{eqn_Lmatrix}
    L_{44}(x_0,x_1) & L_{45}(x_0,x_1) &\cdots& L_{4n}(x_0,x_1)\\
    L_{54}(x_0,x_1) & L_{55}(x_0,x_1) &\cdots& L_{5n}(x_0,x_1)\\
    \vdots & \vdots &\cdots& \vdots\\
    L_{n4}(x_0,x_1) & L_{n5}(x_0,x_1) &\cdots& L_{nn}(x_0,x_1)
\end{bmatrix} 
\end{equation}
are linearly dependent over $\C$. 
   \end{itemize}
We call the matrix of linear forms $S$ \textit{degenerates} if the above condition is satisfied.
\end{definition}

Since first-order data vanish in $P^{n-2}$ defined by $x_2=x_3=0$, the above definition is independent of the choice of coordinate.

%The definition is coordinate free: Let $I_L$ be the ideal of the line $L$ in $X$. Then if $L$ is the 2nd type, and let $h_1,h_2$ be the equations defining the corresponding $P^{n-2}$, then $I_L/(I_L^2,h_1,h_2)$ is trivial, and Definition \ref{def_htL} is about dependence on $I_L^2/(I_L^3,h_1,h_2)$.

 To understand the difference between the two notions, suppose that $L$ is a line of the second type. Then $L$ being a triple line is equivalent to say the symmetric bilinear form associated with the matrix \eqref{eqn_Lmatrix} having a null vector, i.e., there is a non-zero vector $v$ such that $v^TSv=0$. Note that the vector $v$ corresponds to a point $p_v$ in $\mathbb P^n$ with $x_0=\cdots =x_3=0$, and the span of $L$ and $p_v$ is a plane $P^2$ such that $P^2\cap X=3L$ if it is not contained in $X$. On the other hand, $L$ being a higher triple line implies that $S v=0$ and hence $v^TSv=0$ for some non-zero $v$, so we observe that
 \begin{proposition}
    A higher triple line is a triple line.
\end{proposition}

For cubic threefolds, the matrix \eqref{eqn_Lmatrix} is 1-by-1, and being a higher triple line is exactly the same as a triple line, and equivalent to vanishing of $L_{44}(x_0,x_1)$. Starting from cubic fourfolds, being a higher triple line is a more restrictive condition: the existence of a nonzero $v=[v_1,v_2]^T$ such that $v^TSv=0$ does not necessarily imply that the matrix $S$ degenerates: For example, the bilinear form associated with 2-by-2 matrix $$\begin{bmatrix}
    0&L(x_0,x_1)\\L(x_0,x_1)&0
\end{bmatrix}$$ has two null vectors $[1,0]^T$ and $[0,1]^T$, but the matrix itself can be nondegenerate.

%\begin{proof}
%Up to change of coordinates, we can assume the symmetric matrix has vanishing last row and column, so the plane $P$ defined by $x_2=x_3=\cdots=x_{n-1}=0$ is the plane such that $P\cap X=3L$.
%\end{proof}

%This shows it is a natural generalization of the triple line for cubic threefold.

To give a geometrical interpretation, if a line $L$ in $X$ is of the second type, then the intersection $X\cap P^{n-2}$ is singular along the line and has $A_1$ singularity in the tranversal direction at a general point on $L$. If $L$ is further a higher triple line, then $X\cap P^{n-2}$ has $A_2$ singularity in the transversal direction along the entire line $L$. 

\begin{figure}[ht]
\centering
%\begin{equation}
\includegraphics[width=0.25\textwidth]{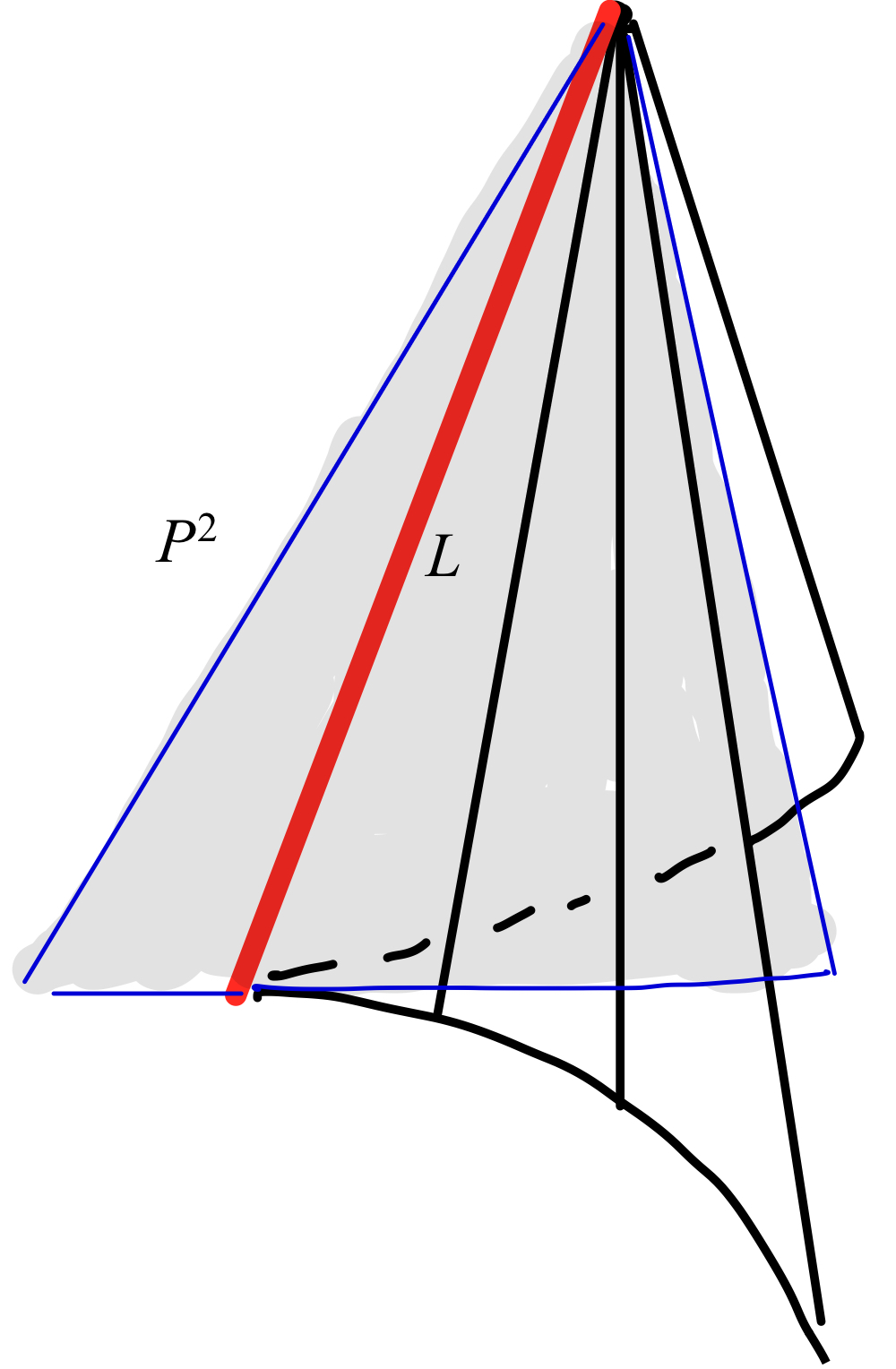}
%\end{equation}
\caption{\label{figure_coneofcusp} $\mathbb P^3$-section of cubic fourfold along a higher triple line}
\end{figure}

For example, a cubic fourfold $X$ with a higher triple line $L$ implies there is a linear $P^3\cong \mathbb P^3$ section of $X$ being a cone over a cuspidal curve.\footnote{When $\dim(X)\ge 5$, $X\cap P^{n-2}$ is not necessarily a cone.} Here, the plane $P^2$ realizing  $P^2\cap X=3L$ is the cone of the "horizontal line" meeting the cuspidal curve at 3 times of the cusp point. In fact, this plane $P^2$ corresponds to a unique vector $v$ up to scaling such that $Sv=0$. In Section \ref{sec_singualarH(X)}, we show how the plane $P^2$ and $P^3$ play a role in describing the singularities of $H(X)$.

We have the following dimension count: 

\begin{proposition}\label{prop_dim-count} (cf. Proposition \ref{prop_generalsmoothdiagonal})
  Having a higher triple line is a condition of codimension at least one in the space of all smooth cubic hypersurfaces. In particular, a general cubic hypersurface has no higher triple line.
\end{proposition}

%(since space of plane cuspidal curves form 7 dimensional space, plus 3 dimensions of choice of cone point.)
The idea is that, for example, when $\dim(X)=4$, the spaces of all cones over cuspidal curves form a codimension 8 subspace of the spaces of cubic surfaces $\mathbb P(Sym^3\C^4)$, while the dimension of the space $Gr(4,6)$ of $\mathbb P^3$ in $\mathbb P^5$ is $8$ and is smaller than $9$. Hence, by incidental correspondence, a general cubic fourfold does not admit a higher triple line. A complete proof will be given in Section \ref{sec_dimcount}. 

%Similar to the finiteness of triple lines on cubic threefold \cite[Lemma 10.15]{CG}, we have

%\subsubsection{A general line of the 2nd type is not a higher triple line}
For special smooth cubic hypersurfaces, we give an upper bound of the dimension of higher triple lines. This generalizes the result on triple lines on cubic threefolds \cite[Lem. 10.15]{CG}.

\begin{proposition}\label{prop_HigherTripleLineinF2}
    Let $X$ be a smooth cubic hypersurface of $\mathbb P^n$ with $n\ge 4$. Then, a general line of the second type (on each irreducible component of $F_2$) is not a higher triple line.
\end{proposition}
\begin{proof}
    We adapt the proof of \cite[Lem. 10.15]{CG} to higher dimensions. Let $C$ be an irreducible component of $F_2$, then as a consequence of \cite[Lem. 7.5]{CG} $W=\bigcup_{t\in C}L_t$ is a subvariety of dimension $n-2$ in $X$. A general point $p_0$ of $W$ is smooth. Since the dual map $\mathcal{D}:X\to X^{\vee}$ is finite, its restriction $\mathcal{D}|_W$ to $W$ must have maximal rank at $p_0$.

  Choose $t_0\in F$ such that $p_0\in L_{t_0}$. The tangent direction $T_{t_0}C$ corresponds to a tangent direction $v\in T_{p_0}W$ normal to $L_0$. Hence, if we choose normal coordinates with respect to the line $L_{t_0}$ so that $X$ has equation \eqref{eqn_cubicnfold@2ndtype}, the codimension-two linear subspace $x_2=x_3=0$ is exactly the tangent space $T_{p_0}W$.

   By assumption that $L_{t_0}$ is a higher triple line, then the $(n+1)\times (n-1)$ submatrix of the Hessian matrix 
\begin{equation}\label{eqn_sub-Hessian}
    \frac{\partial^2F}{\partial x_i\partial x_0}, \frac{\partial^2F}{\partial x_i\partial x_1}, \frac{\partial^2F}{\partial x_i\partial x_4},\cdots, \frac{\partial^2F}{\partial x_i\partial x_n}, i=0,1,\ldots n.
\end{equation}
is degenerate at $p_0$. This shows that the restriction of the dual map $\mathcal{D}|_W$ to $W$ is not of maximal rank at the smooth point $p_0$. This is a contradiction.
\end{proof}

\begin{corollary} \label{cor_htl-dim}
    The dimension of the locus of the higher triple lines is at most $$\dim(F_2)-1=n-4.$$
\end{corollary}
 We will show that this bound is sharp in the case below.

\subsection{Eckardt points}
%A general point $p$ on a cubic hypersurface $X$ has $(\dim(X)-3)$-dimensional family of lines passing through $p$. A point $p\in X$ is called an Eckardt point if there are more lines on $X$ than expected passing through $p$, more precisely, there are $(\dim(X)-2)$-dimensional family of lines passing through $p$. 

\begin{definition}\normalfont
    A point $p\in X$ is called an \textit{Eckardt point} if the tangent hyperplane at $p$ intersects $X$ at $p$ with multiplicity 3.
\end{definition}
Equivalently, $p$ is an Eckardt point if $T_pX\cap X$ is a cone over a smooth cubic hypersurface $Y$ with $\dim(Y)=\dim(X)-2$ and the cone point being the Eckardt point $p$. %$Y$ has dimension $\dim(X)-2$.

There are at most finitely many Eckardt points on a smooth cubic hypersurface \cite[Cor. 6.3.4]{Eckardt-thesis}. As an example, the Fermat cubic $(n-1)$-fold contains $\frac{3n(n+1)}{2}$ Eckardt points and realizes the upper bound \cite[Thm. 3.12]{CCS}. A general cubic hypersurface has no Eckardt point. 

For cubic threefold, each Eckardt point is associated with nine triple lines - the tangent hyperplane section is a cone of a cubic curve $E$. The nine flex points of the cubic curve $E$ correspond to the nine triple lines on $X$. This can be generalized to higher dimensions, except that higher triple lines can vary in a continuous family. We denote $Hess(Y)$ the Hessian hypersurface associated with hypersurface $Y$ defined by $G=0$. It is defined by $\det(\frac{\partial^2G}{\partial x_i\partial x_j})=0$ \cite{BFP-Hessian}.

\begin{proposition}\label{prop_Eckardt}
    Let $p$ be an Eckardt point on $X$, then the higher triple lines on $X$ passing through $p$ is parameterized by the intersection $Y\cap Hess(Y)$ of $Y$ and its associated Hessian hypersurface. In particular, there is a $(n-4)$-dimensional family of higher triple lines on $X$ through $p$.
\end{proposition}
\begin{proof}
    Let $p$ be an Eckardt point, then $T_pX\cap X$ is a cone of a smooth cubic hypersurface $Y$ with the cone point $p$, and satisfies $\dim(Y)=\dim(X)-2$. Then, every line in the ruling is a line of the second type in $X$. Moreover, if point $q\in Y$ such that the tangent hyperplane section $T_qY\cap Y$ has an $A_2$ singularity at $q$, then the line $l_{pq}$ spanned by $q$ and the cone point $p$ is a higher triple line on $X$. Such locus on $Y$ is parameterized by the intersection $Y\cap Hess(Y)$ with the Hessian hypersurface.
\end{proof}

\section{Incidental Subvariety} \label{sec_incidental}
\subsection{Grassmannian} \label{sec_Gr}
In this section, we review the classical theory of Grassmannian and the singularities of the incidence subvariety $D$ of product of Grassmannians. We will show that $\tilde{D}$ is smooth under strict transform of the blow-up of the diagonal. We will find explicit equations of $\tilde{D}$ in an affine coordinate of blow-up of Grassmannians. 

Many of the results are classically known and can be found in \cite{BL00} and \cite{HarrisAG}. Here, we provide a self-contained exposition.

\subsubsection{Lines in $\mathbb P^n$}
Let's first consider the Grassmannian $Gr(2,n+1)$ of lines in $\mathbb P^n$. Parameterizing lines in $\mathbb P^n$ is the same as parameterizing two planes in $\C^{n+1}$ through the origin. So we need two vectors $\vec{v_1}=[a_1,a_2,\ldots ,a_{n+1}]^T$ and $\vec{v_2}=[b_1,b_2,\ldots,b_{n+1}]^T$ that are linearly independent. So, it defines a 2-by-${n+1}$ matrix of rank 2
\begin{equation}
\begin{bmatrix}\label{eqn_2byn+1matrix}
    a_1&a_2&\ldots &a_{n+1}\\
     b_1&b_2&\ldots &b_{n+1}\\
\end{bmatrix}.
\end{equation}

Conversely, (1) any 2-by-$n+1$ matrix determines a 2-dimensional subspace of $\C^4$ if it has rank 2; (2) any two of such matrices determine the same subspace if they differ by a $GL_2(\C)$ action on the left. 

So, to get the one-to-one correspondence between the 2-dimensional subspace of $\C^{n+1}$ and the equivalence classes of 2-by-$n+1$ matrices, we use the Plucker coordinates — consider all of the 2-by-2 minors $x_{ij}=\det\begin{bmatrix}
    a_i&a_j\\
    b_i&b_j
\end{bmatrix}$, where $1\le i<j\le n+1$. Then any two-dimensional subspace corresponds to a point $(x_{ij})$ with Plucker coordinates in $\bbP^{(n+1)n/2-1}$. Place the matrix \eqref{eqn_2byn+1matrix} on top of itself produces a degenerate matrix, and the determinant of its 4-by-4 minor produces the Plucker equations 
\begin{equation}\label{eqn_Plucker}
    x_{ij}x_{st}-x_{is}x_{jt}+x_{it}x_{js}=0,\ 1\le i<j<s<t\le n+1.
\end{equation}

The zero locus defines the Grassmannian $Gr(2,n+1)$. Choosing the affine chart $x_{12}= 1$, we found that the variables $x_{st}$ with $3\le s<t\le n+1$ depend on other factors, and the equations in \eqref{eqn_Plucker} without the term $x_{12}$ are redundant. Hence the open subspace $Gr(2,n+1)\setminus \{x_{12}=0\}$ is isomorphic to $\C^{2(n-1)}$ and has coordinates 
\begin{equation}\label{eqn_Grcoordinate}
    x_{13},x_{14},\ldots, x_{1,n+1},x_{23},x_{24},\ldots,x_{2,n+1}.
\end{equation}

\subsubsection{Incidence variety}
Let $D_L$ be the Schubert subvariety of $Gr(2,n+1)$ that parameterizes the lines in $\bbP^{n+1}$ intersecting a given line $L$. It is normal and irreducible.

We consider a relative version of this cell. 
\begin{definition}\normalfont Let
    $$D\subseteq Gr(2,n+1)\times Gr(2,n+1)$$
be the subvariety parameterizing pairs of incidence lines $D=\{(L_1,L_2)|L_1\cap L_2\neq \emptyset\}$.
\end{definition}

Then, if we let $x_{ij}$, $y_{ij}$ be the Plucker coordinates on the first and the second Grassmannian, then $D$ is given by equations in $Gr(2,n+1)\times Gr(2,n+1)$
\begin{equation}\label{eqn_DGr}
x_{ij}y_{st}-x_{is}y_{jt}+x_{it}y_{js}+x_{st}y_{ij}-x_{jt}y_{is}+x_{js}y_{it}=0,\ 1\le i<j<s<t\le n+1.
\end{equation}
It arises similarly to \eqref{eqn_Plucker} from the Laplace expansion of a 4-by-4 determinant.

\subsubsection{Singularities}
First of all, $D$ is smooth off the diagonal. To see this, $D_L$ has codimension $n-2$ and is only singular at the point $L$, and the projection $p:D\to Gr(2,n+1)$ is a fiber bundle with $p^{-1}(L)\cong D_L$.

As an example, we take $L$ as the line with $x_{12}=1$ and $x_{ij}=0$ for the other $i,j$. Then from \eqref{eqn_DGr}, $D_L$ is given by %$y_{ij}=0$ for $3\le i<j\le n+1$ which is
\begin{equation}\label{eqn_DL}
    y_{1s}y_{2t}=y_{1t}y_{2s},\ s,t\ge 3.
\end{equation}

It has an isolated singularity at the origin. (When $n=3$, it is an ordinary node.) We note that \eqref{eqn_DL} are homogeneous, and the same equations in the projective space define the projective tangent cone at the origin, which we denote by $V_n$. In fact, $V_{n}$ is the Segree embedding
$$\mathbb P^1\times \mathbb P^{n-2}\hookrightarrow \mathbb P^{2n-3}.$$

For example, $V_2$ is $\mathbb P^1$, $V_3$ is the quadric surface. In general, $V_n$ is smooth and $\dim(V_n)=\deg(V_n)=n-1$. Moreover, we can regard $D_L$ as the affine cone of $V_n$, and we have

\begin{lemma}
    The blow-up $\Bl_L(D_L)$ of $D_L$ at the singularity $L$ is smooth, with exceptional divisor $V_n$.
\end{lemma}
%\begin{proof}
%    In an affine chart, the blow-up has equations 
%    \begin{equation*}
%    y_{2t}=y_{1t}y_{23},\  t=4,\ldots, n+1.
%\end{equation*}
%   This defines a smooth locus of codimension $n-2$. 
%\end{proof}

Now, we need a parametric version of this result. To find the equation of $D$ near the diagonal, we choose the affine chart $x_{12}=y_{12}=1$ on both factors. By introducing the diagonal coordinates $u_{ij}=y_{ij}-x_{ij}$ and using \eqref{eqn_Plucker} and \eqref{eqn_DGr}, we find that $D$ in the affine chart is given by
%Hence, $D$ is defined by equation 
%\begin{equation}
%    \begin{cases}
%        x_{st}-x_{1s}x_{2t}+x_{1t}x_{2s}=0,\\
%        y_{st}-y_{1s}y_{2t}+y_{1t}y_{2s}=0,\\
%        y_{st}-x_{1s}y_{2t}+x_{1t}y_{2s}+x_{st}-x_{2t}y_{1s}+x_{2s}y_{1t}=0
%    \end{cases}
%\end{equation}
%for $1\le s<t\le n+1$. By canceling the variable $x_{st}$ and $y_{st}$, we have
%$$y_{1s}y_{2t}-y_{1t}y_{2s}-x_{1s}y_{2t}+x_{1t}y_{2s}+x_{1s}x_{2t}-x_{1t}x_{2s}-x_{2t}y_{1s}+x_{2s}y_{1t}=0$$
\begin{equation}\label{eqn_D}
    u_{1s}u_{2t}=u_{1t}u_{2s},\ s,t\ge 3.
\end{equation}

It is singular along the diagonal $\Delta$ given by $u_{ij}=0$. Consider the blow-up of the diagonal
$$\Bl_{\Delta}\big (Gr(2,n+1)\times Gr(2,n+1)\big )\to (Gr(2,n+1)\times Gr(2,n+1).$$  
Then, in the same way, we have

%\subsection{General situation} Now, consider the Grassmannian of lines in $\mathbb P^n$. The projective variety $Gr(2,n+1)$ has Plucker coordinates $x_{ij}$, with $1\le i<j\le n+1$. It has dimension $2n-2$. 

%In affine chart $x_{12}=1$, $Gr(2,n+1)$ is identified to $\C^{2n-2}$ with coordinates 
%\begin{equation}\label{eqn_Grcoordinate}
%    x_{13},\cdots,x_{1,{n+1}},x_{23},\cdots,x_{2,{n+1}}.
%\end{equation}

\begin{proposition}\label{prop_Dtilde}
    The strict transform $\tilde{D}$ of $D$ is smooth, and has codimension $n-2$.  The exceptional divisor is a $V_n$-bundle over the diagonal.
\end{proposition}

Consequently, Problem \ref{problem} is to study the intersection of two smooth subvarieties.

Note that we can even include the case when $n=2$, $V_2=\mathbb P^1$, as the exceptional divisor of the blowup of $Gr(2,3)\cong (\mathbb P^2)^*$ at a point. So in small dimensions,
\begin{itemize}
    \item $n=2$, $V_2\cong \mathbb P^1$ is a projective line.
    \item $n=3$, $V_3$ is the smooth quadric surface.
   \item $n=4$, $V_4$ is the Segre embedding $\mathbb P^1\times \mathbb P^2\hookrightarrow \mathbb P^5$.
\end{itemize}
%We note that the exceptional divisor of $\tilde{D}$ is a $V_n$-bundle over the diagonal, where $E_{n}$ denotes the subvariety variety defined by $\binom{2n-2}{2}$ quadratic equations \eqref{eqn_D} in the projective space $\mathbb P^{2n-3}$. It is the projectivized tangent cone of $D$ to the normal directions of the diagonal. By the discussion above, $V_n$ is smooth and has dimension $\dim(V_n)=n-1$.

\subsection{Fano variety of lines}\label{sec_TypeIIIsmooth}
In this section,  we study the singularities of the incidental subvariety $D_F$ of $F\times F$. We characterize its singularities and prove that the Hilbert scheme of a pair of skew lines $H(X)$ is smooth away from the diagonal (cf. Corollary \ref{cor_H(X)smoothoffdiagonal}).
%\begin{notation}
%    Let $D_F$ denote the associated reduced scheme of $(F\times F)\cap D$.
%\end{notation}

\begin{definition}\normalfont\label{def_DF}
 Let $D_F$ be the closure in $F\times F$ of the subspace parameterizing pair of incident lines 
    $$I:=\{(L_1,L_2)\in F\times F|L_1\cap L_2\ \textup{is a point}\}.$$
\end{definition}

\begin{proposition}\label{prop_typeIII}
 $D_F$ is irreducible and smooth away from the diagonal.

\end{proposition}
When $\dim(X)=3,4$, it is proved by \cite[Lem. 12.18]{CG} and \cite[Thm. 4.3.1.2]{Franco}. Here, we generalize Giovenzana's argument to higher dimensions.

\begin{proof}
   
We define the $\tilde{I}$ to be the set of triples $\{(L_1,L_2,x)\in F\times F\times X|L_1\cap L_2=\{x\}\}$. Then, the forgetful map
$$\pi:\tilde{I}\to I$$
is an isomorphism.

$\tilde{I}$ can be regarded as a nested Hilbert scheme, and its tangent space is isomorphic to the fiber products
$$T_{(L_1,L_2,x)}\tilde{I}\cong H^0(N_{L_1|X})\times_{N_{L_1|X}(x)}T_xX\times_{N_{L_2|X}(x)} H^0(N_{L_2|X}).$$

In other words, there is a diagram below, where $\phi_i$ is the restriction to the point $x$. $\psi_i$ is the natural projection.

\begin{figure}[ht]
    \centering
\begin{equation}
\begin{tikzcd}
H^0(N_{L_1|X})\arrow[dr,"\phi_1"] && T_xX\arrow[dl,"\psi_1"]\arrow[dr,"\psi_2"]&&H^0(N_{L_2|X})\arrow[dl,"\phi_2"]\\
&N_{L_1|X}(x)&& N_{L_2|X}(x).
\end{tikzcd}
\end{equation}
\end{figure}

Therefore, tangent vectors to $\tilde{I}$ at a pair of incidental lines $(L_1,L_2)$ intersecting at a point $x$ corresponds to two normal vectors $v_i\in H^0(N_{L_i|X})$, with $i=1,2$ and a tangent direction $u\in T_xX$ such that $\phi_1(v_i)=\psi_i(u)$ for both $i=1,2$.

Now we will show $T_{(L_1,L_2,x)}\tilde{I}$ has the expected rank at all points. Note $\dim H^0(N_{L_i|X})=2(n-3)$ no matter the type of the line. $\dim T_xX=n-1$, and $\dim N_{L_i|X}(x)=n-2$. 

We will discuss in the following three cases.

\begin{enumerate}
    \item Both $L_1$ and $L_2$ are lines of the first type. Hence both $\phi_1$ and $\phi_2$ are surjective and $\dim (\ker(\phi_i))=n-4$. So the tangent space is parameterized by vectors $u\in T_xX$, and lifts of $\psi_i(u)$. Therefore, the tangent space $T_{(L_1,L_2,x)}\tilde{I}$ has dimension $n-1+2(n-4)=3n-9$.
    \item  One of the lines (say $L_1$) is of the first type, and the other (say $L_2$) is of the second type. In this case, $\phi_1$ is surjective, and $\phi_2$ has a one-dimensional cokernel, which lifts to a subspace $V_x$ of $T_xX$ which is the tangent space of the unique $(n-2)$-plane (cf. Lemma \ref{lemma_(n-3)plane}). Hence, the tangent space $T_{(L_1,L_2,x)}\tilde{I}$ is parameterized by vectors $u\in V_x$, and the lifts of $\psi_i(u)$ and has dimension $(n-2)+(n-4)+(n-3)=3n-9$.
    \item Both $L_1$ and $L_2$ are lines of the second type. Hence both $\phi_1$ and $\phi_2$ have a one-dimensional cokernel, and their lifts to $T_xX$ are tangent spaces of the $(n-2)$-plane $E_i$.  Then $E_1\neq E_2$. Otherwise, the plane spanned by $L_1$ and $L_2$ will be tangent to $X$ along both $L_1$ and $L_2$, contradicting the fact $X$ has degree three. Hence, the tangent space $T_{(L_1,L_2,x)}\tilde{I}$ is parameterized by vectors $u$ in a $(n-3)$-dimensional subspace $E_1\cap E_2$ of $T_xX$, and the lifts of $\psi_i(u)$, so it has dimension $(n-3)+2(n-3)=3n-9$.
       % \item When $L_1$ and $L_2$ coincide, $E_1=E_2$, hence the tangent space has dimension $(n-2)+2(n-3)=3n-8$.
\end{enumerate}

In summary, the tangent space $T_{(L_1,L_2,x)}\tilde{I}$ has constant dimension everywhere away from diagonal $\Delta_{F}$, hence $\tilde{I}$ and $I$ are smooth. So its closure $D_F$ is irreducible. Since $D_F\setminus I\subseteq \Delta_{F}$, $D_F$ is smooth away from the diagonal.
\end{proof}

The dimension of the tangent space tells us that

\begin{corollary}\label{cor_codimDF}
    $\dim(D_F)=3n-9$. Hence $D_F$ has codimension $n-3$ in $F\times F$, and codimension $n+5$ in $Gr(2,n+1)\times Gr(2,n+1)$.
\end{corollary}

\begin{remark}\normalfont\label{remark_nontransverse}
    One can regard $D_F$ as an irreducible component of $(F\times F)\cap D$ with reduced structure. Note that $\textup{Codim}_{F^2}(D_F)<\textup{Codim}_{Gr(2,n+1)^2}(D)$ (cf. Proposition \ref{prop_Dtilde}, Remark \ref{cor_codimDF}), therefore the intersection is not transverse. 
\end{remark}

\begin{remark}\normalfont\label{remark_Delta&DF}
 When $\dim(X)=3$, the reduced structure of $(F\times F)\cap D$ is $D_F\cup \Delta_F$. The two components intersect along the locus of the lines of the second type on the diagonal.
 
 When $\dim(X)\ge 4$, $D_F$ is the reduced structure of $(F\times F)\cap D$. This is because the normal bundle $N_{L|X}$ at a line of the first type has at least one factor of $\mathcal{O}(1)$ and $D_F\supseteq \Delta_F$.
\end{remark}

We will describe the singularities of $D_F$ in Section \ref{sec_DFsing}.

\subsection{Type (III) schemes} \label{sec_typeIII}
A type (III) subscheme of a cubic hypersurface $X\subseteq \mathbb P^n$ is a closed subscheme $Z_{III}$ of $X$ with Hilbert polynomial $2t+2$ and is a union 
$$Z_{III}=L_1\cup L_2\cup Z_p$$
consisting of a pair of incidental lines $L_1$ and $L_2$ with reduced structure, and $Z_p$ is an embedded point supported at $\{p\}=L_1\cap L_2$ and is contained in a linear 3-plane contained in $T_pX$. When $\dim(X)=3$, one refers to \cite[Lem. 4.6]{YZ_SkewLines}.

Since $D_F\setminus \Delta_F$ parameterizes pairs of incidental lines (with an order), the morphism
$$\widetilde{H(X)}\times_{F\times F} (D_F\setminus \Delta_F)\to D_F\setminus \Delta_F, \ Z_{III}\mapsto L_1\cup L_2$$
as restriction of Hilbert-Chow morphism \eqref{eqn_intro-diagram} forgets the embedded point $Z_p$. The fiber is the projective $(n-4)$-space consisting of the set of all linear 3-plane $P^3$ such that
$$\textup{Span}(L_1,L_2)\subseteq P^3\subseteq T_pX.$$

\begin{corollary}\label{cor_H(X)smoothoffdiagonal}
    The Hilbert scheme of a pair of skew lines $H(X)$ is smooth away from the diagonal, that is, smooth along the locus parameterizing schemes of types (I) and (III).
\end{corollary}
\begin{proof}
 Proposition \ref{prop_typeIII} implies that the blow-up center of $\widetilde{H(X)}\to \Bl_{\Delta_F}(F\times F)$ is smooth away from the diagonal. Hence, $\widetilde{H(X)}$ is smooth away from the diagonal. Since the $\Z_2$ action is free away from the diagonal, its image in the quotient $H(X)$ is also smooth.
\end{proof}

\section{Set-Theoretical Description of  $\textup{Bl}_{\Delta_F}(F\times F)\cap \tilde{D}$} \label{sec_DFtilde}
%In the previous section, we show the Hilbert scheme of a pair of skew lines is smooth away from the diagonal. 

Starting from this section, our goal is to study the singularities of $H(X)$ supported on the diagonal of $\Sym^2F$, which parameterize type (II) and (IV) schemes. As explained in the Introduction, the Hilbert scheme of a pair of skew lines, $H(X)$, up to a double cover, arises from two successive blow-ups, and the second one blows up $\textup{Bl}_{\Delta_F}(F\times F)$ on the closed subscheme $\textup{Bl}_{\Delta_F}(F\times F)\cap \tilde{D}$. Hence to describe $H(X)$, the main question is
\begin{question}
    Is $\textup{Bl}_{\Delta_F}(F\times F)\cap \tilde{D}$ reduced, irreducible? Where are its singularities locate?
\end{question}
We aim to answer these questions in the following sections.

%This can be regarded as pullback of $(F\times F)\cap D$ modulo the exceptional divisor (with multiplicity two). 

%\begin{figure}[ht]
%    \centering
%\begin{equation*}
%\begin{tikzcd}
%\textup{Bl}_{\Delta_F}(F\times F)\cap \tilde{D}\arrow[r,hookrightarrow]\arrow[d] & \textup{Bl}_{\Delta}\big (Gr(2,n+1)\times Gr(2,n+1)\big )\arrow[d]\\
%(F\times F)\cap D\arrow[r,hookrightarrow]& Gr(2,n+1)\times Gr(2,n+1).
%\end{tikzcd}
%\end{equation*}
%\end{figure}

In this section, we use appropriate coordinates to find the defining equations of the scheme-theoretical intersection, as a preparation for the computation to be carried out in the next section. We use the first-order data to give a set-theoretical description of the exceptional locus.
%we will study the intersection $\textup{Bl}_{\Delta_F}(F\times F)\cap \tilde{D}$ set-theoretically. More specifically, we will study the fiber of the intersection over the diagonal $\Delta_F$ for lines of the first and second type and how it varies. Our goal is to show the irreducibility. This is based on the first-order data of a line $L$ in the Fano variety of lines $F$.

\subsection{First order data} Recall that $F$ is the Fano variety of lines on cubic hypersurface $X$, with $\dim(X)=n-1\ge 3$. Then $F\subseteq Gr(2,n+1)$ is a closed subvariety, and local equations can be written with respect to the standard form of the cubic equation \eqref{eqn_cubicnfold@2ndtype} at a given line $L$.  In the local coordinates \eqref{eqn_Grcoordinate}, suppose the line $L$ is of the first type at the origin $x_{ij}=0$, then $F$ has equations \cite[(6.14)]{CG}
\begin{equation}\label{eqn_F1sttype}
\begin{cases}
   x_{13}+\cdots=0.\\
    x_{14}+x_{23}+\cdots=0.\\
    x_{15}+x_{24}+\cdots=0.\\
    x_{25}+\cdots=0.
\end{cases}
\end{equation}

If the line $L$ is of the second type at $x_{ij}=0$, then $F$ has equations \cite[(6.15)]{CG}
\begin{equation}\label{eqn_F2ndtype}
\begin{cases}
   x_{13}+\cdots=0.\\
    x_{14}+\cdots=0.\\
    x_{23}+\cdots=0.\\
    x_{24}+\cdots=0.
\end{cases}
\end{equation}
 Here, $\cdots$ denotes a polynomial that involves terms with orders at least $2$. In this section, we only need the first-order information. One should also note that the first-order data are equivalent to the information from the normal bundle $N_{L|X}$ (cf. Definition \ref{def_1st2ndtype}). We prefer to work with explicit local equations because the computation to be carried out in the next section is based on them.

\subsubsection{Diagonal coordinates}
Now, let $x_{ij}$ and $y_{ij}$ be the coordinates on the first and second factors of Grassmannian. In the affine chart $x_{12}=y_{12}=1$, we introduce diagonal coordinates $u_{ij}=y_{ij}-x_{ij}$ as in Section \ref{sec_Gr}. Then $u_{ij}=0$ defines the diagonal. Now, in the new coordinates $(\underline{u_{ij}},\underline{x_{ij}})=(u_{13},\ldots, u_{2,n+1},x_{13},\ldots, x_{2,n+1})$, a line of the first type $L$ on the diagonal of $F\times F\subseteq Gr(2,n+1)\times Gr(n+1)$ has equations \eqref{eqn_F1sttype} together with
\begin{equation}\label{eqn_F1sttypeU-linear}
\begin{cases}
   u_{13}+\cdots =0,\\
    u_{14}+u_{23}+\cdots=0,\\
    u_{15}+u_{24}+\cdots=0,\\
    u_{25}+\cdots=0.
\end{cases}
\end{equation}
 %Here $\cdots$ is polynomials of $u_{ij}$ and $x_{ij}$ with the lowest degree at least two.

\subsection{Set-theoretical fiber}
Note that to find the set-theoretical fiber of the intersection $\Bl_{\Delta}Gr(2,n+1)^2\cap \tilde{D}$
over the line $(L,L)\in \Delta_F$, we just need to pull back the equations to the blow-up $\Bl_{\Delta}Gr(2,n+1)^2$, intersect with local equations of $\tilde{D}$, and set $x_{ij}=0$. The higher-order terms do not contribute.

%Hence, when we pull back the equations to the blow-up $\Bl_{\Delta}Gr(2,n+1)^2$. Its intersection with $\tilde{D}$ with local equations \eqref{eqn_D} at $x_{ij}=0$ and $u_{ij}=0$ defines a set-theoretical fiber of the intersection $\Bl_{\Delta}Gr(2,n+1)^2\cap \tilde{D}$
%over the line $L\in \Delta_F$ given by $x_{ij}=0$ and $u_{ij}=0$ in the equations, and higher order terms do not contribute. 
Therefore, the fiber over a line of the first type is given by linear parts of \eqref{eqn_F1sttypeU-linear} together with \eqref{eqn_D} in the blowup coordinate $u_{ij}\lambda_{st}=u_{st}\lambda_{ij}$  (cf. \eqref{eqn_blowupcoord}). Direct computation shows it cuts the fiber $V_n$ of $\tilde{D}$ by 
$$\lambda_{1i}=\lambda_{2i}=0,\ i=3,4,5, $$
%is given by (set-theoretically)
%$$  u_{1s}u_{2t}=u_{1t}u_{2s},\ 6\le s<t\le n+1$$
which is isomorphic to $V_{n-3}$.

Similarly, when we consider diagonal coordinates $F\times F$ near a line of the second type, the equation is given by \eqref{eqn_F2ndtype} together with 
\begin{equation}\label{eqn_F2ndtypeU-linear}
\begin{cases}
   u_{13}+\cdots =0,\\
    u_{14}+\cdots=0,\\
    u_{23}+\cdots=0,\\
    u_{24}+\cdots=0.
\end{cases}
\end{equation}

Hence, the fiber of the intersection over a line of second type $L\in \Delta_F$ cuts $V_{n}$ by 
$$\lambda_{13}=\lambda_{14}=\lambda_{23}=\lambda_{24}=0,$$
which is isomorphic to $V_{n-2}$.

%has equations 
%$$  u_{1s}u_{2t}=u_{1t}u_{2s},\ 5\le s<t\le n+1$$
%in $\mathbb P^{2(n-3)}$

%Alternatively, one can work with the normal bundle $N_{L|F}$ based on two types of lines (cf. Definition \ref{def_1st2ndtype}). For $L$ of the first type, the fiber is the intersection of $V_n$ with the $(n-3)$-plane and for the line of the second type, the fiber is the intersection of $V_n$ with the $(n-2)$-plane.

%\begin{proposition}
%    The Jacobian matrix of the equations \eqref{eqn_F1sttypeUBlowup} and \eqref{eqn_D-tilde} at origin is   
%\end{proposition}
\begin{notation}\normalfont
   Let $(\tilde{D}_F):=\textup{Bl}_{\Delta_F}(F\times F)\cap \tilde{D}$ denote the intersection. Let 
$$\Delta(F_2):=\{(L,L)\in D_F\subseteq F\times F\ |\  L\in F_2\}$$
be the diagonal embedding of the locus of lines of the second type.  
\end{notation}

To summarize, we proved

\begin{proposition}\label{prop_DFtilde-fiber}
 The restriction of $(\tilde{D}_F)_{red}$ to the exceptional locus over the diagonal $(\tilde{D}_F)_{red}|_{\Delta_F}$ is a $V_{n-3}$-bundle over the lines of the first type $\Delta_F\setminus \Delta(F_2)$ and a $V_{n-2}$-bundle over the lines of the second type $\Delta(F_2)$.  
    %is irreducible of dimension $3n-9$, and has codimension $n+5$ in $\Bl_{\Delta}(Gr(2,n+1)^2)$.
\end{proposition}
\begin{proof}
    This is based on the above analysis. The only exception is when $n=4$ and $X$ is a cubic threefold, where $(F\times F)\cap D$ is reducible. When we pullback the diagonal equations \eqref{eqn_F1sttypeU-linear} to the blowup, there is another irreducible component $E\subseteq \Bl_{\Delta_F}(F\times F)$, which is the entire exceptional divisor. Remove that, we get the intersection $\textup{Bl}_{\Delta_F}(F\times F)\cap \tilde{D}$ whose restriction to the diagonal is only supported over lines of the second type.
\end{proof}

\begin{figure}[ht]
\centering
%\begin{equation}
\includegraphics[width=0.5\textwidth]{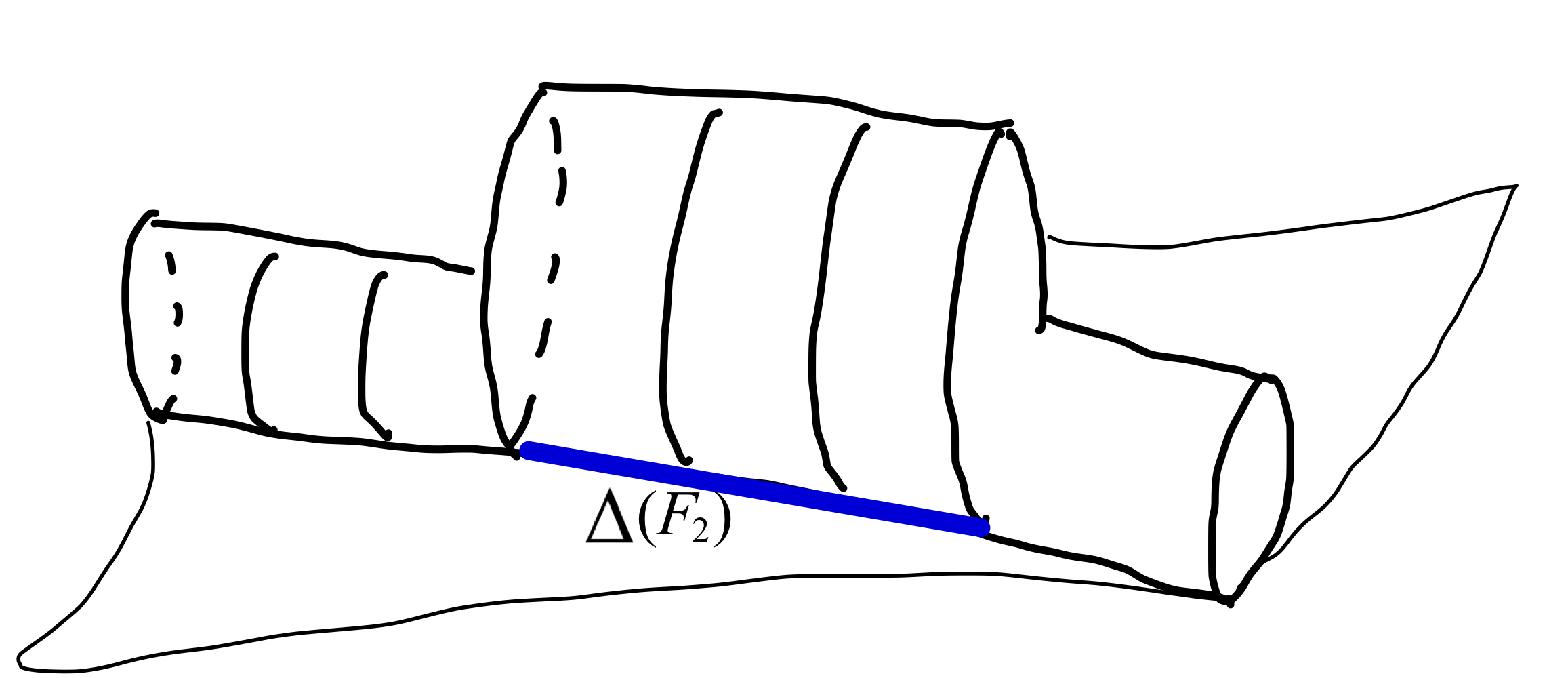}
%\end{equation}
\caption{Set-theoretical Picture of $\textup{Bl}_{\Delta_F}(F\times F)\cap \tilde{D}$}
\end{figure}

\begin{example}\normalfont
 When $X$ is a cubic threefold, $(\tilde{D}_F)_{red}|_{\Delta_F}$ is a $\mathbb P^1$-bundle over $\Delta(F_2)$.  
\end{example}
\begin{example} \normalfont
      When $X$ is a cubic fourfold, $(\tilde{D}_F)_{red}|_{\Delta_F}$ is a $\mathbb P^1$-bundle over the locus of the lines of the first type $\Delta_F\setminus \Delta(F_2)$, and a quadric surface bundle over the second type $\Delta(F_2)$.
\end{example}

\subsection{Type (II) and type (IV) schemes} \label{sec_TypeII&IV}
In this section, we give a modular interpretation of $\tilde{D}_F$ on the diagonal.
\subsubsection{Type (II) scheme}
A type (II) subscheme of a cubic hypersurface $X$ is a closed subscheme $Z_{II}$ of $X$ that arises from a line $L$ and a subline bundle 
$$\mathcal{O}_L\subseteq N_{L|X}.$$

The scheme $Z_{II}$ is the infinitesimal neighborhood of $L$ in the quadric surface "spanned" by $L$ and the normal direction $\mathcal{O}_L$. (When $\dim(X)=3$, one refers to \cite[Prop. 4.9]{YZ_SkewLines}.)

We note that the second blow-up
$\widetilde{H(X)}\to \Bl_{\Delta_F}(F\times F)$
is an isomorphism on the complement of $\tilde{D}_F$. Over the diagonal, this is precisely the Type (II) locus. Hence, we have

\begin{lemma}
    $H(X)$ is smooth on the locus that parameterizes type (II) schemes.
\end{lemma}
\subsubsection{Type (IV) scheme}
A type (IV) subscheme of $X$ is a closed subscheme $Z_{IV}$ with Hilbert polynomial $2t+2$ supported on a line $L$, and admits a primary decomposition 
\begin{equation}\label{eqn_typeIV}
    Z_{IV}=Z_{L,P^2}\cup Z_p,
\end{equation}
where
\begin{itemize}
    \item $Z_{L,P^2}$ is the first-order infinitesimal neighborhood of $L$ in a 2-plane $P^2$ which is tangent to $X$ at all points on $L$;
    \item $Z_p$ is an embedded point whose ideal is the square of $I_{p,P^3}$, the ideal of a reduced point $p\in L$ in a 3-plane $P^3$ tangent to $X$ at $p$.
\end{itemize}

One can filtrate the data of type (IV) scheme \eqref{eqn_typeIV} as follows: (1) the plane $P^2$ corresponds to a subline bundle 
\begin{equation}\label{eqn_O(1)inN_{L|X}}
    \mathcal{O}_L(1)\subseteq N_{L|X},
\end{equation}
and (2) an embedded point corresponds to a 3-plane such that $P^2\subseteq P^3\subset T_pX$.

This filtration corresponds to the factorization of the Hilbert-Chow morphism
$$\widetilde{H(X)}\xrightarrow{\sigma_2} \Bl_{\Delta_F}(F\times F)\xrightarrow{\sigma_1} F\times F,$$
$$Z_{IV}\mapsto (p,L,P^2)\mapsto 2L.$$

For a type (IV) scheme, the intermediate step $\sigma_2$ forgets the embedded point but remembers the normal direction $\mathcal{O}_L(1)$ (and the reduced point). This data is captured at the intersection $(\tilde{D}_F)|_{\Delta_F}$ on the diagonal.

To give a set-theoretical interpretation, Proposition \ref{prop_DFtilde-fiber} says that for each line $L$ of the first type, there are $$\mathbb P^1\times \mathbb P^{n-5}$$
many "intermediate" type (IV) schemes - the $\mathbb P^1$ factor parameterizes $p\in L$, while the second factor parameterizes the embeddings \eqref{eqn_O(1)inN_{L|X}}. When $L$ is a line of the second type, the second factor is switched to $\mathbb P^{n-4}$ due to the change of the normal bundle (see Definition \ref{def_1st2ndtype}).

Note that when $n=4$ and $X$ is a cubic threefold, the fiber is empty over a line a first type. This just means that there is no type (IV) scheme supported on a line of the first type, as observed in \cite[Lemma 4.8]{YZ_SkewLines}.

\section{Tangent Space of $\textup{Bl}_{\Delta_F}(F\times F)\cap \tilde{D}$}\label{sec_general}

In Section \ref{sec_TypeII&IV}, we describe the locus $\textup{Bl}_{\Delta_F}(F\times F)\cap \tilde{D}$ and the exceptional fiber set theoretically based on the first-order data of a line in $X$. %To understand how the fibers vary, we need to study the second-order data.

In this section, we will study the scheme-theoretical information of $\textup{Bl}_{\Delta_F}(F\times F)\cap \tilde{D}$. We will show that it is irreducible and find the condition when it is smooth. This is based on the study of the second-order data of $F$ around a line $L$. We particularly focus on the case where $L$ is a line of the second type.
%we study the singularities of $\tilde{D}_F$ over the diagonal. We prove the reducedness and mention a few related applications.
%Now we assume $n\ge 4$, so the dimension of $X$ is at least 3, and include the discussion in the previous section as a special case.

\subsection{Equations on the blow-up}
Here, we rewrite the diagonal equations \eqref{eqn_F2ndtypeU-linear} of $F\times F$ at a line of second type by including the second-order data.  
%Similarly, the equations for $F\times F$ at a line of the second type is given by \eqref{eqn_F2ndtype} together with 
\begin{equation}\label{eqn_F2ndtypeU}
\begin{cases}
   u_{13}+h_{13}(u_{13},\ldots, u_{2,n+1},x_{13},\ldots, x_{2,n+1})=0.\\
    u_{14}+h_{14}(u_{13},\ldots, u_{2,n+1},x_{13},\ldots, x_{2,n+1})=0.\\
    u_{23}+h_{23}(u_{13},\ldots, u_{2,n+1},x_{13},\ldots, x_{2,n+1})=0.\\
    u_{24}+h_{24}(u_{13},\ldots, u_{2,n+1},x_{13},\ldots, x_{2,n+1})=0.
\end{cases}
\end{equation}
where $h_{ij}$ has degree starting from $2$ and equals to the difference $\phi_{ij}(y_{st})-\phi_{ij}(x_{st})$ of second order terms of \eqref{eqn_F2ndtype}. (Also see \eqref{eqn_phiij} for explicit description.) To express $h_{ij}$ in terms of diagonal coordinates $u_{st}$ and $x_{st}$, we need the following lemma.

%This need to be fised.
\begin{lemma}\label{Lemma_divisiblebyu}
    Let $f(x_1,\ldots,x_m)$ be a homogeneous polynomial of degree $d$. Let $u_i=y_i-x_i$, then
    \begin{equation}\label{eqn_Lemmafx-fy}
      f(y_1,\ldots,y_m)-f(x_1,\ldots,x_m)  
    \end{equation}
is a homogeneous polynomial of degree $d$ in $u_1,\ldots, u_m, x_1,\ldots, x_m$. 
%    has the form 
%    $$\sum_i u_{i}p_{i}+\sum_{i,j} u_{j}u_{k}q_{jk},$$
%    where $p_{i}$ is a polynomial of degree at least one and only involves $x_{1},\ldots, x_{m}$, and $q_{jk}$ is a polynomial.
\end{lemma}
%\begin{proof}Take the Taylor expansion of $f(u_1+x_1,\ldots,u_m+x_m)-f(x_1,\ldots,x_m)$ at $u_1=\cdots=u_m=0$. The result just follows from that the constant term vanishes. 
%\end{proof}
For what follows, the only thing we need from the above is the degree-two terms.

\begin{example}\normalfont \label{example_degree2} Direct computations show that
$$y_{i}^2-x_{i}^2=u_{i}(u_{i}+2x_{i}),$$
    and 
        \begin{align*}
       y_{i}y_{j}-x_{i}x_{j}&=(y_{i}y_{j}-y_{i}x_{j})+(y_{i}x_{j}-x_{i}x_{j}) \\
       &=(u_{i}+x_{i})u_{j}+u_{i}x_{j}\\
       &=u_{j}x_{i}+u_{i}x_{j}+u_{i}u_{j}.
    \end{align*}
\end{example}

%\begin{proof}
%        We prove it by induction on the degree of $f$. When $\deg(f)=2$, it just follows from 
%    $$y_{i}^2-x_{i}^2=u_{i}(u_{i}+2x_{i})$$
%    and 
%    \begin{align*}
%       y_{i}y_{j}-x_{i}x_{j}&=(y_{i}y_{j}-y_{i}x_{j})+(y_{i}x_{j}-x_{i}x_{j}) \\
%       &=(u_{i}+x_{i})u_{j}+u_{i}x_{j}\\
%       &=u_{j}x_{i}+u_{i}x_{j}+u_{i}u_{j}
%    \end{align*}

%Suppose the result is true for $2\le \deg(f)\le k$, then when $\deg(h)=k+1$, each term of \eqref{eqn_Lemmafx-fy} either has the form 

%\begin{align}\label{eqn_gijexpansion1}
%   \nonumber y_{i}^2f(y_{13},\ldots,y_{2,n+1})-x_{ij}^2f(x_{13},\ldots,x_{2,n+1})\\
%    =y_{ij}^2(f(y_{13},\ldots,y_{2,n+1})-f(x_{13},\ldots,x_{2,n+1}))+(y_{ij}^2-x_{ij}^2)f(x_{13},\ldots,x_{2,n+1}))
%\end{align}

%or has the form 
%\begin{align}\label{eqn_gijexpansion2}
%  \nonumber  y_{ij}y_{st}f(y_{13},\ldots,y_{2,n+1})-%x_{ij}x_{st}f(x_{13},\ldots,x_{2,n+1})
%    \\=y_{ij}y_{st}(f(y_{13},\ldots,y_{2,n+1})-f(x_{13},\ldots,x_{2,n+1}))+(y_{ij}y_{st}-x_{ij}x_{st})f(x_{13},\ldots,x_{2,n+1})),
%\end{align}

%where $f$ is a polynomial of degree $k-1$. Then the second term of \eqref{eqn_gijexpansion1} is $u_{ij}(u_{ij}+2x_{ij})f(x_{13},\ldots,x_{2,n+1})$ which satisfies the conclusion. For the first term, by induction, it is $(u_{ij}^2+2u_{ij}x_{ij}+x_{ij}^2)(\sum u_{ab}p_{ab}+\sum u_{ab}u_{cd}q_{abcd})$. Expand it, and it satisfies the conclusion as well. It's a similar analysis for \eqref{eqn_gijexpansion2}.
%\end{proof}

\begin{corollary}\label{Cor_divisiblebyu}
    $h_{ij}=\sum u_{ab}p_{ab}+\sum u_{ab}u_{cd}q_{abcd}$, where $p_{ab}$ is a polynomial of degree at least one and only involves $x_{13},\ldots, x_{2,n+1}$.%, and $q_{abcd}$ is a polynomial. %Similar for $h_{ij}$.
\end{corollary}

\subsubsection{Blowup coordinates}
We blowup the diagonal on $\C^{2n-2}\times \C^{2n-2}$ with coordinate $(u_{ij},x_{ij})$, the equation of the blow-up is
\begin{equation}\label{eqn_blowupcoord}
    (u_{13},\ldots,u_{2,n+1},x_{13},\ldots, x_{1,n+1};\lambda_{13},\ldots, \lambda_{2,n+1})\in \C^{2n-2}\times \C^{2n-2}\times \bbP^{2n-3}, u_{ij}\lambda_{st}=u_{st}\lambda_{ij}.
\end{equation}

Choose an affine chart $\lambda_{2,n+1}=1$, and set $u=u_{2,n+1}$, then we have 
\begin{equation}\label{eqn_blowupcoordinate}
    u_{ij}=\lambda_{ij}u.
\end{equation}

Pullback the equation \eqref{eqn_F2ndtypeU} to the blow-up and applying the substitution \eqref{eqn_blowupcoordinate}, we find by Corollary \ref{Cor_divisiblebyu} that $u$ can be factored out. Hence we set $\bar{h}_{ij}=h_{ij}/u$, and we obtain the defining equations of the strict transform $\textup{Bl}_{\Delta_F}(F\times F)$ of $F\times F$ over a line of the second type on the diagonal 
\begin{equation}\label{eqn_F2ndtypeUBlowup}
\begin{cases}
   \lambda_{13}+\bar{h}_{13}(u\lambda_{13},\ldots, u\lambda_{2,n},u,x_{13},\ldots, x_{2,n+1})=0,\\
   \lambda_{14}+\bar{h}_{14}(u\lambda_{13},\ldots, u\lambda_{2,n},u,x_{13},\ldots, x_{2,n+1})=0,\\
    \lambda_{23}+\bar{h}_{23}(u\lambda_{13},\ldots, u\lambda_{2,n},u,x_{13},\ldots, x_{2,n+1})=0,\\
    \lambda_{24}+\bar{h}_{24}(u\lambda_{13},\ldots, u\lambda_{2,n},u,x_{13},\ldots, x_{2,n+1})=0.
\end{cases}
\end{equation}

To compute $\bar{h}_{ij}$, we need to use Example \eqref{example_degree2}. If the second-order term of one of the equations in \eqref{eqn_F2ndtype} includes $x_{ab}x_{cd}$, then the corresponding $h_{ij}$ contains the term $u_{ab}x_{cd}+u_{cd}x_{ab}+u_{ab}u_{cd}$. Substituting using \eqref{eqn_blowupcoordinate}, $\bar{h}_{ij}$ includes the term
\begin{equation}\label{eqn_hijbarterm}
\lambda_{ab}x_{cd}+\lambda_{cd}x_{ab}+\lambda_{ab}\lambda_{cd}u.
\end{equation}
These calculations will be used in Section \ref{sec_SolveEquations} to determine the rank of the Jacobian matrix.

\subsection{Jacobian matrix}
To summarize the equations we found, we look at the defining equations of the intersection $\textup{Bl}_{\Delta_F}(F\times F)\cap \tilde{D}$, restricted to the preimage of the diagonal under the blow-up map. We have three sets of equations
\begin{align}
&\lambda_{s}+\bar{h}_{s}(u\lambda_{13},\ldots, u\lambda_{2,n},u,x_{13},\ldots, x_{2,n+1})=0, \label{eqn_cubicnfold1}\\
&x_{s}+\phi_{s}(x_{13},\ldots,x_{2,n+1})=0. \label{eqn_cubicnfold2}
%\begin{cases}
%   \lambda_{13}+\bar{h}_{13}(u\lambda_{13},\ldots, u\lambda_{2,n},u,x_{13},\ldots, x_{2,n+1})=0.\\
%   \lambda_{14}+\bar{h}_{14}(u\lambda_{13},\ldots, u\lambda_{2,n},u,x_{13},\ldots, x_{2,n+1})=0.\\
%    \lambda_{23}+\bar{h}_{23}(u\lambda_{13},\ldots, u\lambda_{2,n},u,x_{13},\ldots, x_{2,n+1})=0.\\
%    \lambda_{24}+\bar{h}_{24}(u\lambda_{13},\ldots, u\lambda_{2,n},u,x_{13},\ldots, x_{2,n+1})=0.
%\end{cases}
\end{align}
for subscript $s=13, 14, 23, 24$, defining $\textup{Bl}_{\Delta_F}(F\times F)$, together with the pullback of \eqref{eqn_D} 
\begin{equation}\label{eqn_cubicnfold3}
 \lambda_{1i}-\lambda_{1,n+1}\lambda_{2i}=0,\ i=3,\ldots,n
\end{equation}
defining $\tilde{D}$.

Take the partial derivatives of these equations and evaluate at $u=x_{13}=\cdots=x_{2,n+1}=0$ on the diagonal at the given line $L$ of the second type. The parameters $\lambda_{ij}$ are allowed to take any value on the fiber, which is described in Proposition \ref{prop_DFtilde-fiber}. We obtain the Jacobian matrix \eqref{eqn_generalJacobian1}.

\begin{equation}\label{eqn_generalJacobian1}
\begin{blockarray}{ccccccc|cccc|cccccccc}
\lambda_{13} & \lambda_{14} & \lambda_{23} & \lambda_{24}&\lambda_{15}&\cdots& \lambda_{1n} & x_{13} & x_{14} & x_{23} & x_{24} & \lambda_{1,n+1} & u & x_{15} & x_{25}&\cdots \\
\begin{block}{(ccccccc|cccc|ccccccc)c}
  1 & 0 & 0 & 0 & 0 & \cdots& 0 & * & * & * & * & 0 &\frac{\partial \bar{h}_{13}}{\partial u} & \frac{\partial \bar{h}_{13}}{\partial x_{15}} & \frac{\partial \bar{h}_{13}}{\partial x_{25}}&\cdots \\
  0 & 1 & 0 & 0 & 0 & \cdots & 0 & * & * & * & * & 0 & \frac{\partial \bar{h}_{14}}{\partial u} & \frac{\partial \bar{h}_{14}}{\partial x_{15}} & \frac{\partial \bar{h}_{14}}{\partial x_{25}}&\cdots \\
  0 & 0 & 1 & 0 & 0 & \cdots & 0 & * & * & * & * & 0 & \frac{\partial \bar{h}_{23}}{\partial u} & \frac{\partial \bar{h}_{23}}{\partial x_{15}} & \frac{\partial \bar{h}_{23}}{\partial x_{25}}&\cdots \\
  0 & 0 & 0 & 1 & 0 & \cdots & 0 & * & * & * & * & 0 & \frac{\partial \bar{h}_{24}}{\partial u} & \frac{\partial \bar{h}_{24}}{\partial x_{15}} & \frac{\partial \bar{h}_{24}}{\partial x_{25}}&\cdots \\
\cline{1-16}
  0 & 0 & 0 & 0 & 0 & \cdots & 0 & 1 & 0 & 0 & 0 & 0 & 0 & 0 & 0 &\cdots \\
  0 & 0 & 0 & 0 & 0 & \cdots & 0 & 0 & 1 & 0 & 0 & 0 & 0 & 0 & 0 &\cdots \\
  0 & 0 & 0 & 0 & 0 & \cdots & 0 & 0 & 0 & 1 & 0 & 0 & 0 & 0 & 0 &\cdots \\
  0 & 0 & 0 & 0 & 0 & \cdots & 0 & 0 & 0 & 0 & 1 & 0 & 0 & 0 & 0 &\cdots \\
  \cline{1-16}
   1 & 0 & -\lambda_{1,n+1} & 0 & 0 & \cdots & 0 & 0 & 0 & 0 & 0 & -\lambda_{23} & 0 & 0 & 0 &\cdots \\
    0 & 1 & 0 & -\lambda_{1,n+1} & 0 & \cdots & 0 & 0 & 0 & 0 & 0 & -\lambda_{24} & 0 & 0 & 0 &\cdots \\
    0 & 0 & 0 & 0 & 1 & \cdots & 0 & 0 & 0 & 0 & 0 & -\lambda_{25} & 0 & 0 & 0 &\cdots \\
     \vdots & \vdots & \vdots & \vdots &  & \ddots &  & 0 & 0 & 0 & 0 & \vdots & \vdots & \vdots & \vdots & \\
      0 & 0 & 0 & 0 & 0 & \cdots & 1 & 0 & 0 & 0 & 0 & -\lambda_{2n} & 0 & 0 & 0 &\cdots \\
\end{block}
\end{blockarray}
\end{equation}

It is a matrix with $n+6$ rows and $4n-4$ columns.
Note that here we discarded the $\lambda_{25},\ldots, \lambda_{2n}$ columns since they do not affect the rank computation. Also, when $n=4$, the columns 5 through $n$-th should be deleted.

The rows of the Jacobian matrix span the cotangent space at a point $p$ on the underlying reduced scheme of the intersection $\textup{Bl}_{\Delta_F}(F\times F)\cap \tilde{D}$.

\begin{proposition}\label{prop_Jrank-2nd-type}
   The Jacobian matrix \eqref{eqn_generalJacobian1} has rank at least $n+4$ and at most $n+5$. 
\end{proposition}
\begin{proof}
The first eight rows and the last $n-4$ rows are clearly linearly independent. Their span has dimension $n+4$ and is contained in the cotangent space.
The rank of the matrix is at most the codimension of the intersection $\tilde{D}_F$ in $\Bl_{\Delta}(Gr(2,n+1)^2)$, which is at most $n+5$ (cf. Corollary \ref{cor_codimDF}, Proposition \ref{prop_DFtilde-fiber}).
\end{proof}

\subsubsection{Non-transversal intersection}
From another point of view, the equations \eqref{eqn_cubicnfold1}, \eqref{eqn_cubicnfold2} and the last $n-4$ equations in \eqref{eqn_cubicnfold3} has linearly independent tangent vectors, and therefore they intersect transversely at $u=x_{ij}=0$. As a result, the variety they define is smooth and has codimension equal to the number of equations.

However, upon adding the two additional equations
$$\lambda_{13}-\lambda_{1,n+1}\lambda_{23}=0\ \textup{and}\  \lambda_{14}-\lambda_{1,n+1}\lambda_{24}=0,$$
the intersection is no longer transverse — the dimension of variety decreases only by 1 (this already occurs before the blow-up, as observed in Remark \ref{remark_nontransverse}). Therefore the Jacobian matrix has rank at most $n+5$ and cannot attain full row rank.

Meanwhile, the dimension of the tangent space decreases by either $1$ or $0$. In the former case, the variety defined by \eqref{eqn_cubicnfold1}, \eqref{eqn_cubicnfold2}, and \eqref{eqn_cubicnfold3} remains smooth locally, and the scheme structure is reduced (cf. \cite[Lem. 5.1]{Li}). In the latter case, the underlying variety is singular, and the scheme may have a non-reduced structure. Therefore, we need to analyze when the second case occurs. This depends on the second-order data.

Before proceeding with the row reduction computation, let us mention a similar setup to compute Jacobian matrix over a line of the first type, which turns out to be completely determined by the first-order data.
%\begin{lemma} \cite[Lemma 5.1]{Li} \label{lemma_cleanintersection}
%    Let $M$ be a smooth variety, and let $X_1$ and $X_2$ be two smooth subvarieties. Then the scheme theoretic intersection $X_1\cap X_2$ is smooth if and only if  they intersect cleanly, i.e., $(X_1\cap X_2)_{red}$ is smooth.
%\end{lemma}

\subsubsection{Line of the 1st type}
We can also analyze the Jacobian matrix for $\textup{Bl}_{\Delta_F}(F\times F)\cap \tilde{D}$ above a line of the first type on the diagonal. Similarly, equations \eqref{eqn_F1sttype}, the pullback of diagonal equations \eqref{eqn_F1sttypeU-linear}, and the last $n-5$ equations of \eqref{eqn_cubicnfold3} intersect transversely. The remaining three equations 
$$\lambda_{1i}-\lambda_{1,n+1}\lambda_{2i}=0,\ \textup{with} \ i=3,4,5$$
further cut down the dimension of the variety by $2$. It suffices to understand how the normal vector associated with the above three polynomials linearly depends on the normal vectors associated with the pullback of diagonal equations \eqref{eqn_F1sttypeU-linear}.
The direct calculation shows that the relevant variables are $\lambda_{13},\lambda_{14}, \lambda_{15}$, and $\lambda_{23},\lambda_{24}, \lambda_{25}$, and the corresponding 7 by 6 matrix always has the maximal rank. Hence, the whole Jacobian matrix always has the expected rank.
\begin{proposition}\label{prop_Jrank-1st-type}
    The Jacobian matrix at point of $\textup{Bl}_{\Delta_F}(F\times F)\cap \tilde{D}$ over a line of the first type over the diagonal always has expected rank $n+5$.
\end{proposition}

Let us highlight a direct consequence of what we know about the Jacobian matrix so far.
\subsubsection{Irreducibility} 
\begin{lemma}\label{lemma_irreducible}
    The scheme-theoretical intersection $\tilde{D}_F=\Bl_{\Delta_F}(F\times F)\cap \tilde{D}$ is irreducible. 
\end{lemma}
\begin{proof}
 First, the strict transform of an incidental variety $D_F$ is an irreducible component of the intersection $\Bl_{\Delta_F}(F\times F)\cap \tilde{D}$. Suppose that there is a second component, say $D'$. Then, it must be supported on the exceptional locus over the diagonal, and its intersection with the strict transform is a singular point. By Proposition \ref{prop_Jrank-1st-type}, it is not supported over lines of the first type. Hence, it has to be supported on the locus above the lines of the second type. However, by Lemma \ref{lemma_2ndtypedim} and Proposition \ref{prop_DFtilde-fiber}, such locus has dimension 
 $$\dim(F_2)+\dim(V_{n-2})=(n-3)+(n-3)=2n-6.$$
 Hence, the codimension of $D'$ is at least $\dim(\Bl_{\Delta}Gr(2,n+1)^2)-(2n-6)=2n+2$, which exceeds $n+5$, the rank of the Jacobian matrix (cf. Proposition \ref{prop_Jrank-2nd-type}), which is a contradiction.
% (cf. Proposition \ref{prop_Jrank-2nd-type}), if there is a second irreducible component $C$, then its general point has the dimension of the tangent space being the same dimension as the main component, which is the strict transform of $D_F$. So $\dim(C)=\dim(D_F)=3n-9$, and they intersect along the singular locus with dimension at most $n-3$ (cf. Theorem \ref{thm_summarySec6}). On the other hand, the intersection of the two components have intersection dimension at least 
%$$(3n-9)+(3n-9)-2(2n-6)=2n-6>n-3,$$
\end{proof}

\subsubsection{Row Reductions}
Now we come back to the Jacobian matrix \eqref{eqn_generalJacobian1} for a line of the second type, and we want to determine when its rank drops. This reduces to understanding the linear relation between the 9th and 10th row of the Jacobian matrix \eqref{eqn_generalJacobian1} and the first four rows. By subtracting the first row from the 9th row and plus $\lambda_{1,n+1}$ times the third row, we can eliminate the left-hand-side entries of the 9th row. A similar process can be applied to the 10th row. Thus, the condition for a rank drop is determined by the vanishing of the remaining entries in the 9th and 10th rows.

\begin{proposition}
 The Jacobian matrix \eqref{eqn_generalJacobian1} has (lower) rank $n+4$ if and only if
\begin{equation}\label{eqn_Jacobian1}
    \lambda_{23}=\lambda_{24}=0,
\end{equation}
together with 
\begin{equation}
\begin{cases}\label{eqn_Jacobian2}
    \lambda_{1,n+1}\frac{\partial \bar{h}_{2j}}{\partial u}-\frac{\partial \bar{h}_{1j}}{\partial u}=0,\\
    \lambda_{1,n+1}\frac{\partial \bar{h}_{2j}}{\partial x_{i,k+1}}-\frac{\partial \bar{h}_{1j}}{\partial x_{i,k+1}}=0,
\end{cases}
\end{equation}
for $i=1,2$, $j=3,4$, and $4\le k\le n$.
All expressions are interpreted as being evaluated $u=x_{ij}=0$.
\end{proposition}

\subsection{Second order data}
To solve the above equations, we need to work out the computations based on the equations of $\bar{h}_{ij}$. In light of Lemma \ref{Lemma_divisiblebyu} and \eqref{eqn_hijbarterm}, this amounts to understanding the second-order terms of $\phi_{ij}$ in \eqref{eqn_cubicnfold2}. Let's find the explicit equations for $\phi_{ij}$.

Recall that if $L$ is a line of the second type in $X$, then by changing the coordinates we can assume $L$ is defined by equation $x_2=\cdots=x_{n}=0$ and the cubic hypersurface $X$ as \eqref{eqn_cubicnfold@2ndtype}

\begin{equation}
   F(x_0,\ldots,x_n)=x_2x_0^2+x_3x_1^2+\sum_{2\le i,j\le n}x_ix_jL_{ij}(x_0,x_1)+C(x_2,\cdots,x_n),
\end{equation}
where $L_{ij}(x_0,x_1)=L_{ji}(x_0,x_1)$ is a linear homogeneous polynomial, and $C(x_2,\ldots,x_n)$ is a homogeneous cubic.

In the Plucker coordinates, recall we choose the affine chart $x_{12}=1$, and the coordinates $(x_{13},x_{14},\ldots, x_{1,n+1},x_{23},x_{24},\ldots, x_{2,n+1})$ corresponds to the $\C^2$ spanned by row vectors of the matrix
$$
\begin{bmatrix}
    1&0&x_{13}&\ldots &x_{1,n+1}\\
     0&1&x_{23}&\ldots&x_{2,n+1}\\
\end{bmatrix}.
$$

Therefore, for a line $L'$ to be contained in $X$ and near $L$, it satisfies the equations 
$$F(\lambda[1:0:x_{13}:\ldots:x_{1,n+1}]+\mu[0:1:x_{23}:\ldots:x_{2,n+1}])=0$$
for all $[\lambda:\mu]\in \mathbb P^1$. Use \eqref{eqn_cubicnfold@2ndtype}, the above equation gives
\begin{align*}
0&=\lambda^2(\lambda x_{13}+\mu x_{23})+\mu^2(\lambda x_{14}+\mu x_{24})\\
&+\sum_{2\le i,j\le n+1}(\lambda x_{1,i+1}+\mu x_{2,i+1})(\lambda x_{1,j+1}+\mu x_{2,j+1})L_{ij}(\lambda,\mu)+\textup{third order terms.}
%&+(\lambda x_{13}+\mu x_{23})^2L_{22}(\lambda,\mu)+(\lambda x_{14}+\mu x_{24})^2L_{33}(\lambda,\mu)+\cdots+(\lambda x_{1,n+1}+\mu x_{2,n+1})^2L_{nn}(\lambda,\mu)\\
%&+(\lambda x_{13}+\mu x_{23})(\lambda x_{14}+\mu x_{24})L_{23}(\lambda,\mu)+(\lambda x_{13}+\mu x_{23})(\lambda x_{15}+\mu x_{25})L_{24}(\lambda,\mu)\\
%&+\cdots+(\lambda x_{1n}+\mu x_{2n})(\lambda x_{1,n+1}+\mu x_{2,n+1})L_{n-1,n}(\lambda,\mu)
\end{align*}

 Expanding in terms of $\lambda^3,\lambda^2\mu,\lambda\mu^2,\mu^3$, we get four equations as in \eqref{eqn_F2ndtype}, we write the second-order terms below.

\begin{align}\label{eqn_phiij}
\phi_{13}&=\sum_{4\le i,j\le n}a_{ij}^0x_{1,i+1}x_{1,j+1},\nonumber \\%+a_{23}^0x_{13}x_{14}+a_{24}^0x_{13}x_{15}+\cdots+a_{n-1,n}^0x_{1n}x_{1,n+1}. \\
\phi_{23}&=\sum_{4\le i,j\le n}a^1_{ij}x_{1,i+1}x_{1,j+1}+\sum_{4\le i, j\le n}a^0_{ij}(x_{1,i+1}x_{2,j+1}+x_{1,j+1}x_{2,i+1}), \nonumber \\
%2a_{44}^0x_{15}x_{25}+\cdots+2a_{nn}^0x_{1,n+1}x_{2,n+1}\\
%&+a_{23}^1x_{13}x_{14}+\sum_{4\le k\le n}(a_{2k}^1x_{13}x_{1,k+1}+a_{2,k}^0x_{13}x_{2,k+1})+\sum_{4\le k\le n}(a_{3k}^1x_{14}x_{1,k+1}+a_{3k}^0x_{14}x_{2,k+1})\\
\phi_{14}&=\sum_{4\le i,j\le n}a^0_{ij}x_{2,i+1}x_{2,j+1}+\sum_{4\le i,j\le n}a^1_{ij}(x_{1,j+1}x_{2,i+1}+x_{1,i+1}x_{2,j+1}),\\%+2a_{44}^1x_{15}x_{25}+\cdots+2a_{nn}^1x_{1,n+1}x_{2,n+1}\\
%&+\sum_{4\le k\le n}(a_{2k}^0x_{13}x_{2,k+1}+a_{3k}^1x_{14}x_{2,k+1})
\phi_{24}&=%a_{44}^1x_{25}^2+\cdots+a_{nn}^1x_{2,n+1}^2+
\sum_{4\le i, j\le n}a_{ij}^1x_{2,i+1}x_{2,j+1}.\nonumber 
\end{align}

Here $L_{ij}(x_0,x_1)=a_{ij}^0x_0+a_{ij}^1x_1$. Note that for convenience, we drop the third-order term because they do not affect the computation of the tangent space, and we drop the terms that involve $x_{13},x_{14},x_{23}$ and $x_{24}$ for the reason below.

\subsection{Solving the equations}\label{sec_SolveEquations}
Now we look at the conditions given by equations \eqref{eqn_Jacobian1} and \eqref{eqn_Jacobian2}.

First, using \eqref{eqn_cubicnfold3} and $\lambda_{23}=\lambda_{24}=0$, we obtain $\lambda_{13}=\lambda_{14}=0$ in the blow-up coordinates. So, the terms in $\phi_{ij}$ that involve $x_{13},x_{14},x_{23}$ and $x_{24}$ are irrelevant.

Second, note $\phi_{24}$ does not depend on $x_{1j}$, so in particular, $\frac{\partial \bar{h}_{24}}{\partial x_{1j}}=0$, and the equations $\lambda_{1,n+1}\frac{\partial \bar{h}_{24}}{\partial x_{1,k+1}}-\frac{\partial \bar{h}_{14}}{\partial x_{1,k+1}}=0$  from \eqref{eqn_Jacobian2}  reduces to
\begin{equation}\label{eqn_2nequtions1}
    \frac{\partial \bar{h}_{14}}{\partial x_{1,k+1}}|_{u=x_{ij}=0}=0,\ k=4,\ldots, n
\end{equation}
Use \eqref{eqn_hijbarterm} and \eqref{eqn_phiij}, we find %the terms of $\bar{h}_{14}$ that involve $x_{1,k+1}$ are
$$\bar{h}_{14}=\sum_{4\le i,j\le n}a_{ij}^0(2x_{2,i+1}\lambda_{2,j+1}+u\lambda_{2,i+1}\lambda_{2,j+1})+2a_{ij}^1(x_{1,j+1}\lambda_{2,i+1}+x_{2,i+1}\lambda_{1,j+1}+u\lambda_{1,i+1}\lambda_{2,j+1}).$$

Hence \eqref{eqn_2nequtions1} reduces to
\begin{equation}\label{eqn_2nequtions2}
   2\sum_{4\le i\le n}a^1_{ik}\lambda_{2,i+1}=0,\ k=4,\ldots, n
\end{equation}

%However, the term $2a_{kk}^1x_{1,k+1}x_{2,k+1}$ in $\phi_{14}$ corresponds to terms 
%$$2a_{kk}^1(\lambda_{1,k+1}x_{2,k+1}+x_{1,k+1}\lambda_{2,k+1}+u\lambda_{1,k+1}\lambda_{2,k+1})$$
%in $\bar{h}_{14}$, so \eqref{eqn_2nequtions1} implies 
%\begin{equation}\label{eqn_2nequtions2}
%   a_{kk}^1\lambda_{2,k+1}=0,\ k=4,\ldots, n.
%\end{equation}
%Since $\lambda_{2,n+1}=1$ is our affine chart, we have $a_{nn}^1=0$.

Similarly, $\phi_{13}$ does not depend on $x_{2j}$, so in particular, $\frac{\partial \bar{h}_{13}}{\partial x_{2j}}=0$, and the equations $ \lambda_{1,n+1}\frac{\partial \bar{h}_{23}}{\partial x_{2,k+1}}-\frac{\partial \bar{h}_{13}}{\partial x_{2,k+1}}$ from \eqref{eqn_Jacobian2} reduces to 
\begin{equation}\label{eqn_2nequtions3}
    \lambda_{1,n+1}\frac{\partial \bar{h}_{23}}{\partial x_{2,k+1}}|_{u=x_{ij}=0}=0,\ k=4,\ldots, n
\end{equation}
By a similar computation, we find
\begin{equation}\label{eqn_2nequtions4}
  2\lambda_{1,n+1}\sum_{4\le i\le n}a^0_{ik}\lambda_{1,i+1}=0,\ k=4,\ldots, n
\end{equation}

Now, use equations \eqref{eqn_cubicnfold3} and assume that $\lambda_{1,n+1}\neq 0$, the equations \eqref{eqn_2nequtions2} and \eqref{eqn_2nequtions4} become %(here we identify $a_{ij}^p=a_{ji}^p$)
\begin{equation}\label{eqn_2nequtions5}
    \begin{cases}
       \sum_{4\le i\le n}a^1_{ik}\lambda_{2,i+1}=0\\
        \sum_{4\le i\le n}a^0_{ik}\lambda_{2,i+1}=0\\
    \end{cases}
\end{equation}
for all $k=4, \ldots, n$.

In addition, equations $\lambda_{1,n+1}\frac{\partial \bar{h}_{24}}{\partial x_{2,k+1}}-\frac{\partial \bar{h}_{14}}{\partial x_{2,k+1}}=0$ from \eqref{eqn_Jacobian2} reduce to

\begin{equation}\label{eqn_2nequtions8}
 \lambda_{1,n+1}(2\sum_{4\le i\le n}a^1_{ik}\lambda_{2,i+1})-(2\sum_{4\le i\le n}a^0_{ik}\lambda_{2,i+1}+2\sum_{4\le i\le n}a^1_{ik}\lambda_{1,i+1})=0,\ k=4,\ldots, n.
\end{equation}
Together with \eqref{eqn_cubicnfold3} and \eqref{eqn_2nequtions2}, they implies \eqref{eqn_2nequtions5} under the assumption of $\lambda_{1,n+1}=0$.

One can check the rest of the equations in \eqref{eqn_Jacobian2} are dependent on the equations above and do not provide new relations.

In summary, our computation shows that

\begin{proposition}\label{prop_Jmatrixdegenerate}
    The Jacobian matrix \eqref{eqn_generalJacobian1} at point of $\textup{Bl}_{\Delta_F}(F\times F)\cap \tilde{D}$ over a line of the second type on the diagonal has lower rank $n+4$ if and only if the two symmetric matrices
\begin{equation}\label{eqn_2nequtions6}
A^p:=\begin{bmatrix}
    a_{44}^p & a_{45}^p &\cdots& a_{4n}^p\\
    a_{54}^p & a_{55}^p &\cdots& a_{5n}^p\\
    \vdots & \vdots &\cdots& \vdots\\
    a_{n4}^p & a_{n5}^p &\cdots& a_{nn}^p
\end{bmatrix},\ p=0,1 
\end{equation}
are degenerate, and the kernel of the two matrices has a nontrivial intersection, which is equivalent to the matrix of linear forms \eqref{eqn_Lmatrix} 
$$S=A^0x_0+A^1x_1$$
being degenerate.
\end{proposition}

When $n=4$ and $X$ is cubic threefold, this is exactly $$a_{44}^0=a_{44}^1=0,$$ 
and equivalent to $L$ being a triple line on $X$ (cf. Definition \ref{def_tripleline}). %We will discuss the geometric meaning of \eqref{eqn_2nequtions6} being singular for higher dimensions in the next section.

\begin{remark}\normalfont\label{remark_pv}
Suppose $v=[\lambda_{25},\ldots,\lambda_{2,n+1}]^T$ is a non-zero vector that lies in the common kernel of $A^p$. Then $Sv=0$ and it corresponds to a point $p_v\in \mathbb P^n$ "at infinity" with coordinates $x_i=0$ for $0\le i\le 3$. The span of $p_v$ and $L$ determines a plane $P^2(v)$. If $P^2(v)$ is not contained in $X$, then it makes $L$ a triple line — $P^2(v)\cap X=3L$.
\end{remark}

\subsection{Summary}

To summarize our computation carried out in this section, we proved:

\begin{theorem}\label{thm_summarySec6}
    Let $X\subseteq \mathbb P^n$ be a smooth cubic hypersurface with $n\ge 4$. Let $p$ be a point on $Bl_{\Delta}(F\times F)\cap \tilde{D}$ over a line $(L,L)$ on the diagonal. 
    \begin{itemize}
        \item  If $L$ is of the first type, then the tangent space at $p$ always has expected dimension $3n-9$;
        \item If $L$ is of the second type, then the tangent space at $p$ has expected dimension $3n-9$ if and only if the matrix of linear forms \eqref{eqn_Lmatrix} is non-degenerate, i.e., $L$ is not a higher triple line.
    \end{itemize}
\end{theorem}
\begin{proof}
    The argument about lines of the first (resp. second) type follows from Proposition \ref{prop_Jrank-1st-type} (resp. \ref{prop_Jmatrixdegenerate}). 
\end{proof}

\begin{remark}\normalfont\label{remark_DFtilde-exceptional}
    One can also show that the intersection $\tilde{D}_F\cap E$ is smooth if and only if $X$ has no higher triple line, where $E$ is the exceptional divisor of $\Bl_{\Delta_F}(F\times F)$. This can be obtained by keeping track of the computation of rank of the Jacobian matrix \eqref{eqn_generalJacobian1} with an additional condition $u=0$.
\end{remark}

We also provide an upper bound of the dimension of singularities of $\Bl_{\Delta_F}(F\times F)\cap \tilde{D}$.

\begin{lemma}\label{lemma_dimSing}
   Let $X$ be a smooth cubic hypersurface in $\mathbb P^n$ with $n\ge 4$. The singular locus of $\tilde{D}_F:=\Bl_{\Delta_F}(F\times F)\cap \tilde{D}$ has dimension at most $2n-7$.
\end{lemma}
\begin{proof}
  By Theorem \ref{thm_summarySec6}, the singular locus of $\tilde{D}_F$ is over the locus of higher triple lines on the diagonal. By Corollary \ref{cor_htl-dim}, the locus of lines of higher triple lines has dimension at most $n-4$. Since the fiber dimension is at most $n-3$ (cf. Proposition \ref{prop_DFtilde-fiber}), we have 
  $$\dim(Sing(\tilde{D}_F))\le (n-4)+(n-3)=2n-7.$$
\end{proof}

The bound is sharp when $\dim(X)=3$, where the singular locus consists of finitely many $\mathbb P^1$ over finitely many triple lines. When $\dim(X)\ge 4$, we don't know if the bound is still sharp, but for the purpose of proving the reducedness below, it is enough.

\section{Reducedness of $\textup{Bl}_{\Delta_F}(F\times F)\cap \tilde{D}$}\label{sec_reducedness}

In this section, we study the reducedness properties of the intersection $\textup{Bl}_{\Delta_F}(F\times F)\cap \tilde{D}$. In particular, Theorem \ref{thm_summarySec6} tells us where the intersection is smooth by computing the rank of the Jacobian matrix. On the locus where the Jacobian matrix has lower rank, the scheme may have an embedded component. However, we will show this cannot happen.

\begin{lemma}\label{lemma_pdtrick}
    Let $A$ be a regular local ring over a field $k$, and $I\trianglelefteq A$ be an ideal generated by two elements $I=(f,g)$. Then, the depth of the quotient ring $\textup{depth}(A/I)$ is at least $\dim(A)-2$. 
    
    %Suppose the quotient local ring $R/I$ satisfies
   % \begin{itemize}
   %     \item $\dim R/I=\dim R-1$. 
        %\item $R/I$ is not regular.
   % \end{itemize}

\end{lemma}
\begin{proof}
   By assumption, $A$ is a UFD. Hence, we can write $f=hf_1$ and $g=hg_1$, where $h$ is the greatest common divisor of $f$ and $g$.%, and neither $f_1$ nor $f_2$ is a unit in $R$. Otherwise, the

We have a free $A$-module resolution of $A/I$
$$A\xrightarrow{\begin{bmatrix}g_1\\-f_1\end{bmatrix}} A^2\xrightarrow{[f\ g]}A\to A/(f,g).$$

Hence, the projective dimension of $A/I$ is at most $2$. Therefore, by Auslander–Buchsbaum formula \cite[Theorem 19.9]{Eisenbud-CA}, 
$$\textup{depth}(A/I)=\textup{depth}(A)-\textup{pd}(A/I).$$

    Since $A$ is regular, $\dim(A)=\textup{depth}(A)$, the claim follows.
\end{proof}

\begin{remark}\normalfont \label{rmk_depth-1}
    $\textup{depth}(A/I)=\dim(A)-1$ if and only if $f_1$ or $f_2$ is a unit, if and only if $f$ and $g$ are multiple of each other, if and only if $R/I$ is a (local) complete intersection, i.e., a hypersurface.
\end{remark}

In what follows, we apply the lemma to a local ring localized at a regular affine subvariety of $\Bl_{\Delta_F}(F\times F)\cap \tilde{D}$ containing singular locus. For example, when $\dim(X)=3$, $\Bl_{\Delta_F}(F\times F)\cap \tilde{D}$ is 4-dimensional and its singular locus is disjoint union $\sqcup_i\mathbb P^1$ indexed by finitely triple lines, so $A$ is a 3-dimensional regular local ring, arising from the localization of coordinate ring of an affine chart of $\Bl_{\Delta_F}(F\times F)\cap \tilde{D}$ at one of the $\mathbb P^1$. The lemma ensures $A/I$ has depth is at least one, allowing us to conclude reducednss using Serre's criterion.

\begin{proposition}\label{prop_cubicnfold_reduced}
  Let $X$ be a smooth cubic hypersurface $\mathbb P^n$ with $n\ge 4$. Then the scheme-theoretical intersection $\Bl_{\Delta_F}(F\times F)\cap \tilde{D}$ is reduced.
\end{proposition}
\begin{proof}
By Theorem \ref{thm_summarySec6}, it suffices to study a neighborhood on the exceptional locus over $(L,L)$ on the diagonal, where $L$ is a line of the second type.

In an affine chart, the ideal $J$ of the intersection $\Bl_{\Delta_F}(F\times F)\cap \tilde{D}$ is generated by the 8 functions \eqref{eqn_cubicnfold1} and \eqref{eqn_cubicnfold2} cutting out $\Bl_{\Delta_F}(F\times F)$, together with the $n-2$ functions \eqref{eqn_cubicnfold3} cutting out $\tilde{D}$. Our goal is to show the reducedness of $R/J$, where $R$ is the polynomial ring with variables $u$, $x_{ij}$, and $\lambda_{ij}$.

%$J_1$ of $\Bl_{\Delta_F}(F\times F)$ is generated by the 8 equations \eqref{eqn_cubicnfold1} and \eqref{eqn_cubicnfold2}; The ideal $J_2$ of $\tilde{D}$ is generated by the $n-2$ equations \eqref{eqn_cubicnfold3}.  

%Let $R'=\C[x_{ij},\lambda_{ij},u]/J_1$ denote the coordinate ring of $\Bl_{\Delta_F}(F\times F)$ in the chart and extend the ideal $J_2$. Hence, the intersection is given by the ring $R'/J_2$.

Recall that according to the computation of the Jacobian matrix \eqref{eqn_generalJacobian1}, the equations \eqref{eqn_cubicnfold1}, \eqref{eqn_cubicnfold2} and the last $n-4$ equations in \eqref{eqn_cubicnfold3} intersect transversely. Let $J_1$ denote the ideal generated by these functions. Hence  $R':=R/J_1$ is a regular ring, which has dimension $$\dim(R')=\dim \Bl_{\Delta_F}(Gr(2,n+1)^2)-8-(n-4)=3n-8.$$

Moreover, $R\to R/J$ factors through $R'\to R'/J_2$, where $J_2$ is the ideal generated by the remaining two functions
$$f:=\lambda_{13}-\lambda_{1,n+1}\lambda_{23}, g:=\lambda_{14}-\lambda_{1,n+1}\lambda_{24}.$$

Now, we have the isomorphism $R/J\cong R'/J_2$, so after localization, we are in the situation in Lemma \ref{lemma_pdtrick}.

To show the reducedness of $R'/J_2$, we need to show Serre's criterion \cite[Ex. 11.10]{Eisenbud-CA}
 $$(\mathcal{R}_0)+(\mathcal{S}_1).$$

 First, Lemma \ref{lemma_dimSing} implies that $\textup{Spec}(R'/J_2)$ is smooth in codimension one, so in particular, it satisfies $(\mathcal{R}_0)$. For $(\mathcal{S}_1)$, it suffices to check localization at prime ideals whose support is contained in the singular locus.
 %localization of $R/I$ at the prime ideals whose support are contained in the singular locus of $\textup{Spec}(R/I)$.

%We lift the prime ideal to a prime $\mathfrak{p}$ is a prime ideal of $R$
Specifically, let $\mathfrak{p}$ be a prime ideal of $R'$ containing $J_2$.  Now apply Lemma \ref{lemma_pdtrick} to the regular local ring $R'_{\mathfrak{p}}$ and the ideal $(J_2)_{\mathfrak{p}}$, so we obtain
\begin{equation}\label{eqn_depth}
    \depth(R'_{\mathfrak{p}}/(J_2)_{\mathfrak{p}})\ge \dim (R'_{\mathfrak{p}})-2.
\end{equation}

It takes the minimal value when the prime ideal $\mathfrak{p}$ corresponds to an irreducible component of the singular locus of $\textup{Spec}(R'/J_2)$. According to Lemma \ref{lemma_dimSing}, the dimension of the singular locus is bounded above by $2n-7$. Therefore, 
$$\dim(R'_{\mathfrak{p}})=\textup{height}(\mathfrak{p})\ge (3n-8)-(2n-7)=n-1.$$ Hence by \eqref{eqn_depth}, we find that $\depth(R'_{\mathfrak{p}}/(J_2)_{\mathfrak{p}})\ge n-3\ge 1$, as long as $n\ge 4$. Therefore, $R'/J_2$ satisfies Serre's condition $(\mathcal{S}_1)$ and is reduced.
%The lift of regular sequence of $R'$, together with \eqref{eqn_regularseq}, defines a regular sequence of $R$ of length $1+(n-4)=n-3$. Hence $R$ satisfies the Serre condition $S_{n-3}$.
\end{proof}

\begin{corollary} \label{cor_H(X)normal}
$H(X)$ is normal.
\end{corollary}
\begin{proof}
 By passing to the $\Z_2$ quotient, Proposition \ref{prop_cubicnfold_reduced} implies that the birational morphism $$H(X)\to \Bl_{\Delta_F}\Sym^2F$$ is a blow-up of a smooth scheme along a reduced subscheme. The ideal sheaf of a reduced scheme is integrally closed \cite[Rmk. 1.1.3 (4)]{Irena}. Therefore, by \cite[Prop. 5.2.4]{Irena}, the corresponding Rees algebra is integrally closed. Hence, the blow-up is normal. 
\end{proof}

\begin{remark} \normalfont
    According to \cite[Lem. 10.21]{Kollar}, Proposition \ref{prop_cubicnfold_reduced} implies that the union $$\Bl_{\Delta_F}(F\times F)\cup \tilde{D},$$
with the reduced scheme structure is \textit{seminormal}.
\end{remark}

\section{Main Theorems and Applications}\label{sec_app}
In this section, we mention some applications as consequences of results in previous sections. In particular, Theorem \ref{thm_H(X)3fold} and \ref{thm_mainprecise} from the introduction will be proved.

\subsection{Cubic Threefolds}
We note that when $X$ is a cubic threefold, $\tilde{D}_F$ has codimension one in $\Bl_{\Delta_F}(F\times F)$. Hence, by Proposition \ref{prop_cubicnfold_reduced} and Lemma \ref{lemma_irreducible}, it is a divisor and, moreover, a Cartier divisor due to the smoothness of $\Bl_{\Delta_F}(F\times F)$. Since blow-up at a Cartier divisor is an isomorphism, we have $\widetilde{H(X)}\cong \Bl_{\Delta_F}(F\times F)$. Passing to the $\Z_2$ quotient, we recover the main theorem in \cite{YZ_SkewLines}.
%even $\tilde{D}_F$ is singular when the cubic threefold has a triple line, but just do not contribute to the second blowup since it has codimension one. 

\begin{corollary}(\cite{YZ_SkewLines}) \label{cor_cubic3fold-Hsmooth}
    Let $X$ be a smooth cubic threefold. Then, the Hilbert scheme of a pair of skew lines $H(X)$ is smooth and isomorphic to $\Bl_{\Delta_F}\Sym^2F$.
\end{corollary}
%\begin{proof}
%    By Proposition \ref{prop_cubicnfold_reduced}, $\Bl_{\Delta_F}(F\times F)\cap \tilde{D}=\tilde{D}_F$ is always reduced. When $n=4$, $\tilde{D}_F$ has codimension one in the smooth variety $\Bl_{\Delta_F}(F\times F)$, so it is a Cartier divisor. Since the blow-up of a smooth variety at a Cartier divisor is an isomorphism, we have that $\widetilde{H(X)}\cong \Bl_{\Delta_F}(F\times F)$ is smooth. Passing to the $\Z_2$ quotient, we have the smoothness of $H(X)$.
%\end{proof}
For higher dimensions, $\tilde{D}_F$ has codimension at least two, so the singularities of $\tilde{D}_F$ will contribute to the singularities of the second blow-up.

\subsection{Smoothness Criterion of Hilbert scheme} 

\begin{theorem}(cf. Theorem \ref{thm_mainprecise}) \label{thm_mainprecisebody}
Let $X$ be a cubic hypersurface with $\dim(X)\ge 4$. Then $H(X)$ is smooth if and only if the $X$ has no higher triple line.
\end{theorem}
\begin{proof}
According to the analysis in the Introduction, to show the smoothness of $\widetilde{H(X)}$, it suffices to show that $\textup{Bl}_{\Delta_F}(F\times F)\cap \tilde{D}$ is set-theoretically smooth. By Proposition \ref{prop_typeIII}, it is always smooth away from the diagonal, as long as $X$ is smooth. By Theorem \ref{thm_summarySec6}, when $X$ has no higher triple lines, the second blow-up center $\textup{Bl}_{\Delta_F}(F\times F)\cap \tilde{D}$ is smooth everywhere. Hence, the blow-up $\widetilde{{H(X)}}$ is smooth. Now, the double cover $\widetilde{{H(X)}}\to H(X)$ is branched along the strict transform $\tilde{E}$ of the exceptional divisor $E$ of $\textup{Bl}_{\Delta_F}(F\times F)$. On the other hand, $\tilde{E}\to E$ is blow-up of the intersection $\tilde{D}_F\cap E$, which is again smooth when $X$ has no higher triple line (cf. Remark \ref{remark_DFtilde-exceptional}), hence $\tilde{E}$ is a smooth divisor and the $\Z_2$-quotient $H(X)$ is smooth as well.

   Conversely, when $X$ has a higher triple line, blowup center $\widetilde{{H(X)}}\to \Bl_{\Delta_F}(F\times F)$ is reduced and singular of codimension at least two, so $\widetilde{H(X)}$ and its $\Z_2$ quotient is singular as well.
\end{proof}

\subsection{Hypersurface singularities}
Recall that a variety $Y$ has hypersurface singularities if $\dim(T_pY)\le \dim(Y)+1$ at any point $p\in Y$. Equivalently, $Y$ can be locally analytically embedded in $\C^{\dim(Y)+1}$ as a hypersurface. 
\begin{corollary}\label{cor_reducedness}
Let $X$ be a smooth cubic hypersurface in $\mathbb P^n$ with $n\ge 4$. Then the scheme-theoretical intersection $\tilde{D}_F=\Bl_{\Delta_F}(F\times F)\cap \tilde{D}$ is reduced, irreducible, and has hypersurface singularities. 
\end{corollary}
\begin{proof}
The reducedness and irreducibility follow from Proposition \ref{prop_cubicnfold_reduced} and Lemma \ref{lemma_irreducible}. 
Finally,  Proposition \ref{prop_Jrank-2nd-type} tells us that the rank of the Jacobian matrix can drop by at most one. Hence, the tangent space has dimension at most $\dim(\tilde{D}_F)+1$, showing $D_F$ has hypersurface singularities.
\end{proof}

\begin{remark}\normalfont
    Having hypersurface singularities implies that $\tilde{D}_F$ is a local complete intersection and is Cohen-Macaulay. By Lemma \ref{lemma_dimSing}, $\tilde{D}_F$ is smooth in codimension $(2n-6)\ge 2$, therefore $\tilde{D}_F$ is normal.
\end{remark}

\begin{proposition}\label{prop_H(X)tilde-hypersing}
    The branched double cover $\widetilde{H(X)}$ of the Hilbert scheme of a pair of skew lines of a smooth cubic hypersurface $X$ has hypersurface singularities. 
\end{proposition}
\begin{proof}
   By Corollary \ref{cor_reducedness}, the birational morphism $\widetilde{H(X)}\to \Bl_{\Delta_F}(F\times F)$ is a blow-up of smooth variety at a subvariety with hypersurface singularities. Hence the blow-up center is a local complete intersection. Moreover, as in the proof of Proposition \ref{prop_cubicnfold_reduced}, locally, the blow-up center is a hypersurface of a smooth subvariety $Z'$ of an open subspace $\Bl_{\Delta_F}(F\times F)$. Hence, the blow-up is locally defined by one equation, and has hypersurface singularities.
\end{proof}

\subsection{Singularities of incidental subvariety} \label{sec_DFsing}
Recall that $D_F \subseteq F\times F$ is the incidental subvariety (cf. Definition \ref{def_DF}). Now, as a consequence of the reducedness and irreducibility (cf. Corollary \ref{cor_reducedness}), we have a fibered diagram.

\begin{figure}[ht]
    \centering
\begin{equation*}
\begin{tikzcd}
 \tilde{D}_F\arrow[r,hookrightarrow]\arrow[d,"\sigma_F"] & \textup{Bl}_{\Delta_F}(F\times F)\arrow[d,"\sigma"]\\
D_F\arrow[r,hookrightarrow]& F\times F.
\end{tikzcd}
\end{equation*}
\end{figure}

In other words, $\tilde{D}_F$ is the strict transform of $D_F$. Moreover, $\sigma_F$ is a desingularization if $X$ has no higher triple line. We can use the birational morphism $\sigma_F:\tilde{D}_F\to D_F$ to characterize the singularities of $D_F$. 

%\begin{proposition}\label{prop_DFnormal}
%  $D_F$ is normal. 
%\end{proposition}
%\begin{proof}
%As a consequence of Proposition \ref{prop_DFtilde-fiber} and Corollary \ref{cor_reducedness}, the birational morphism $\tilde{D}_F\to D_F$ has connected fibers. Since $\tilde{D}_F$ is normal, by \cite[Lem. 10.15]{Kollar}, $D_F$ is also normal.
%\end{proof}

\begin{proposition}\label{prop_DFsing}
Let $(D_F)^{\textup{sing}}$ denote the singular locus of $D_F$.

\begin{enumerate}
    \item When $\dim(X)=3$, $(D_F)^{\textup{sing}}$ is the locus of triple lines on the diagonal;
    \item When $\dim(X)= 4$,  $(D_F)^{\textup{sing}}=\Delta(F_2)$ is the locus of the lines of the second type;
    \item When $\dim(X)\ge 5$, $(D_F)^{\textup{sing}}=\Delta_F$ is the whole diagonal. 
\end{enumerate}
\end{proposition}
\begin{proof}
(1) and (2) follow from \cite[Lem. 12.18]{CG}, \cite{triplelines} and \cite[Thm. 4.3.1.2]{Franco}. 

We assume $\dim(X)=n-1\ge 4$. Hence $\Delta_F\subseteq D_F$ by Remark \ref{remark_Delta&DF}. The fiber of $\tilde{D}_F\to D_F$ over a line of the first type is $V_{n-3}$, the Segre embedding $\mathbb P^1\times \mathbb P^{n-5}\hookrightarrow \mathbb P^{2n-9}$ (cf. Proposition \ref{prop_DFtilde-fiber}). It is a nondegenerate variety, so the tangent space of $D$ at $(L, L)$ normal to the diagonal is the ambient vector space $\C^{2n-8}$. Since $2n-8+\dim(\Delta_F)\ge \dim(D_F)$ (cf. Corollary \ref{cor_codimDF}), and the equality holds iff $n=5$. We conclude that $D_F$ is smooth at $(L, L)$ when $\dim(X)=n-1=4$ and singular otherwise. Since lines of the first type is dense in $F$,  this implies (3). A similar analysis at the line of 2nd type implies (2).
%Similarly, when $L\in F_2$ is a line of the second type, the fiber of $\tilde{D}_F\to D_F$ at $(L, L)$ is the Segre embedding $V_{n-2}$, which spans the ambient space $\C^{2n-6}$. When $L$ is a smooth point of $F_2$, such space is the tangent space of $D$ at $(L, L)$ normal to the diagonal. Therefore, $\dim(T_{(L,L)}D_F)=2n-6+\dim(\Delta_F)>\dim (D_F)$, and $D_F$ is singular there.
\end{proof}

\section{Dimension Count}\label{sec_dimcount}
%\subsection{Dimension Count and Geometric Interpretation}

This section is a continuation of Section \ref{sec_highertriplelines}. We will bound the dimension of the space of cubic hypersurfaces with a higher triple line and prove Proposition \ref{prop_dim-count}, which indicates the main Theorem \ref{thm_main}.

\subsection{Transversal $A_2$ Singularities}
We want to understand the geometric meaning of the degeneracy condition of the matrix of linear forms \eqref{eqn_Lmatrix}.

The linear dependence of the columns implies that we can choose new coordinates $x_4,\ldots,x_n$ such that the two matrices \eqref{eqn_2nequtions6} have the last row being zero (hence the last column is zero as well). In other words, in equation \eqref{eqn_cubicnfold@2ndtype} of the cubic hypersurface, there is no term that involves the monomials
$$x_px_ix_n,\ p=0,1,\ \textup{and}\  i=4,\ldots,n.$$

This implies that the codimension two linear section of the cubic hypersurface $X$ given by $P^{n-2}=\{x_2=x_3=0\}$ is a cubic $(n-3)$-fold with equation
\begin{equation}\label{eqn_cubictransversalA2}
    \sum_{4\le i,j\le n-1}x_ix_jL_{ij}(x_0,x_1)+C(x_4,\cdots,x_n)=0.
\end{equation}

Now, the codimension-two subvariety $Y=X\cap P^{n-2}$ is singular along the line $L_{x_0,x_1}$ and has \textit{transversal $A_2$ singularities} along the line.
%Hence the codimension-two subvariety $M\cap X$ is singular along the line $L$, with $\det(L_{ij})=0$. We will call condition \eqref{eqn_cubictransversalA2} for $M\cap X$ to have \textit{transversal $A_2$ singularities}. 

For example, when $\dim(X)=3$, $P^2\cap X=3L$ is a triple line; when $\dim(X)=4$, $P^3\cap X$ is a cone over a cuspidal plane curve (cf. figure \ref{figure_coneofcusp}). %In higher dimensions, any $\mathbb P^3$-section of $M\cap X$ containing the line $L$ is a cone over a cuspidal plane curve.

To provide a dimension count of cubic hypersurfaces with a higher triple line, we first need to count the dimension of the cubics of the form \eqref{eqn_cubictransversalA2}. We first need a linear algebra argument:

\begin{lemma}\label{lemma_Sdet=0}
    Suppose that the matrix $S$ degenerates (cf. Definition \ref{def_htL}), then $\det(S)=0$ as a homogeneous polynomial in $x_0,x_1$.  
\end{lemma}
\begin{proof}
By changing the coordinates described above, one can assume that the matrix of the linear forms \eqref{eqn_Lmatrix} has vanishing last row and column. Therefore, $\det(S)=0$. Alternatively, a coordinate-free proof can be given using Cramer's rule \cite[XIII, Thm. 4.4]{Lang}. %Suppose there is a linear relation 
   % $$\sum_i\lambda_{i}S_i=0,$$
   % where $\lambda_i\in \C$ and $S_i$ the $i$-th coloumn of $S$. Then for each $i$, $\lambda_i\det(S)=0$ in the ring of the homogeneous polynomials on $\mathbb P^1$. Since the ring is integral and does not have a zero divisor, if one of $\lambda_i\neq 0$, $\det(S)$ must be zero.
\end{proof}

It is not clear to us if the converse holds, but Lemma \ref{lemma_Sdet=0} at least provides an estimation of the number of relations. For example, if $S$ has size $(n-3)\times (n-3)$, then $\det(S)$ is a homogeneous polynomial of degree $n-3$, therefore $\det(S)=0$ provides $n-2$ conditions.
%\subsection{Dimension count}

\begin{lemma}\label{lemma_transversalA2count}
    The space of cubic hypersurfaces in $\mathbb P^{n-2}$ with transversal $A_2$ singularities along a line forms a subspace of codimension at least $2n-1$.
\end{lemma}
\begin{proof}
%For convenience, we first identify $\mathbb P^{n-2}$ as a subspace of $\mathbb P^n$ defined by $\{x_2=x_3=0\}$.

%We count the dimensions of the cubic $n-3$-folds with 
Let $Y\in \mathbb P(Sym^3\mathbb C^{n-1})$ be a cubic hypersurface in $\mathbb P^{n-2}$. We require $Y$ to have transversal $A_2$ singularities along a line. Then it imposes the following conditions
\begin{enumerate}
    \item $Y$ contains the line $L=\{x_2=\cdots=x_{n-2}=0\}$ requires the vanishing of the monomials $x_0^3,x_0^2x_1,x_0x_1^2,x_1^3$ in the defining equation of $Y$, which gives four conditions.
    \item $Y$ is singular along $L$ forces the vanishing of the terms $l_{ij}(x_2,\ldots,x_{n-2})x_0^ix_1^j$, with $i+j=2$ and $l_{ij}$ a linear form. Hence, it gives $3(n-3)$ conditions.
    \item The matrix of linear forms $S$ degenerates implies $\det(S)=0$ (cf. Lemma \ref{lemma_Sdet=0}), which provides another $n-2$ conditions.
\end{enumerate}

 Therefore, (1), (2), and (3) together give $4+3(n-3)+(n-2)=4n-7$ conditions. We allow the line $L$ to move in $\mathbb P^{n-2}$. It is parameterized by $Gr(2,n-1)$, which has dimension $2(n-3)$, so the space of cubic hypersurfaces that have transversal $A_2$ singularities along a line is a closed subspace of $\mathbb P(Sym^3\mathbb C^{n-1})$ with codimension at least
$$(4n-7)-2(n-3)=2n-1.$$
\end{proof}

%One should think that $Y$ in Lemma \ref{lemma_transversalA2count} arises from cutting a 2-dimensional higher smooth cubic hypersurface by a linear subspace along a higher triple line. 

\begin{remark}\normalfont
    The bound in Lemma \ref{lemma_transversalA2count} is sharp when $\dim(X)=3$ and 4.
\end{remark}

\begin{example}\normalfont
\begin{itemize}
    \item In $\mathbb P^2$, a cubic curve with transversal $A_2$ along a line is three times of a line $3L$, nonreduced. They form a codimension 7 subspace of all cubic curves.
    \item In $\mathbb P^3$, a cubic surface with transversal $A_2$ along a line is a cone over a cuspidal cubic curve, which is non-normal (cf. Remark \ref{remark_hierarchy}). Such cubic surfaces form a codimension 9 subspace of all cubic surfaces.
    \item In $\mathbb P^4$, a cubic threefold with transversal $A_2$ singularities along a line is normal (if there are no other singularities). Such cubics form a subspace of codimension (at least) 11 of all cubic threefolds.
\end{itemize}
\end{example}

\subsection{Correspondence} Now we use intersection correspondence to bound the dimension of the cubic hypersurface with a higher triple line.
\begin{proposition}\label{prop_generalsmoothdiagonal}
    The set of smooth cubic hypersurfaces in $\mathbb P^n$, with $n\ge 4$, which have higher triple lines, forms a closed subspace of codimension at least one. %In particular, a general cubic hypersurface does not satisfy condition \textbf{(P)}, and according to Theorem \ref{thm_summarySec6}, the set-theoretical intersection $Bl_{\Delta}(F\times F)\cap \tilde{D}$ is smooth over the diagonal.
\end{proposition}
\begin{proof}
Let $W$ be the space of all cubic $(n-3)$-fold in $\mathbb P^n$. Then $\pi: W\to Gr(n-1,n+1)$ is a $\mathbb P(Sym^3\mathbb C^{n-1})$-bundle, with fiber over $P^{n-2}\subseteq \mathbb P^n$ being the spaces of cubic hypersurfaces in $P^{n-2}$. Then, by Lemma \ref{lemma_transversalA2count}, and the bundle structure, the locus $\mathcal{C}$ consisting of cubic $(n-3)$-fold in $\mathbb P^n$ with transversal $A_2$ singularities along a line has codimension at least $2n-1$.

We consider the incidence map
$$\phi:\mathbb P(Sym^3\C^{n+1})\times Gr(n-1,n+1)\to W,$$
$$(X,P^{n-2})\mapsto X\cap P^{n-2}, $$
by intersecting a cubic hypersurface $X$ and a codimension-two plane $P^{n-2}$ in $\mathbb P^n$. 

Then $\phi$ is subjective. We use the fact that the space of $\mathbb P^{n-2}$ in $\mathbb P^n$ has dimension $\dim Gr(n-1,n+1)=2(n-1)$. Therefore, $pr_1(\phi^{-1}(\mathcal{C}))$ has codimension at least $$(2n-1)-2(n-1)=1$$ in the space of all cubic hypersurfaces in $\mathbb P^n$. Hence, the claim follows.
\end{proof}

\begin{corollary} (cf. Theorem \ref{thm_main})
Let $X$ be a general cubic hypersurface with $\dim (X)\ge 3$. The Hilbert scheme of a pair of skew lines $H(X)$ is smooth. 
\end{corollary}
\begin{proof}
  This follows from Theorem \ref{thm_mainprecisebody} and Proposition \ref{prop_generalsmoothdiagonal}.
\end{proof}

\section{Interpreting Singularities of the Hilbert Scheme}\label{sec_singualarH(X)}
In this section, we provide a modular meaning of the singularities of $\tilde{D}_F$ and $H(X)$. %Note that by Corollary \ref{cor_cubic3fold-Hsmooth}, $H(X)$ is nonsingular when $\dim(X)=3$. So, we assume $\dim (X)\ge 4$ in this section. %Specifically, we will study how the singularity on $\tilde{D}_F$ find in Proposition \ref{prop_Jmatrixdegenerate} and \ref{thm_summarySec6} contribute to the singularities on $H(X)$ via the second blow-up.   

To summarize what we have proved, by descending the first column of \eqref{eqn_intro-diagram} to the $\Z_2$ quotient, $H(X)$ arises from successive blowup
$$\Bl_{\tilde{D}_F'}\Bl_{\Delta_F}\Sym^2F\to \Bl_{\Delta_F}\Sym^2F\to \Sym^2F$$
along the diagonal $\Delta_F$, and $\tilde{D}_F'$, which is the $\Z_2$ quotient of $\tilde{D}_F$. %When $X$ has a higher triple line $L$, $\tilde{D}_F$ has a one-dimensional singular locus isomorphic to $\mathbb P^1$ over $L$ on the diagonal.

From Proposition \ref{prop_Jmatrixdegenerate}, we see that the singular locus of $\tilde{D}_F$ is on the diagonal fixed by $\Z_2$ action, hence we can identify the singular locus of $\tilde{D}_F'$ and $\tilde{D}_F$, which is fibered over locus of higher triple lines on the diagonal $\Delta_F$ and the fiber is at least a copy of $\mathbb P^1$. Then, according to the description of type (IV) schemes in Section \ref{sec_TypeII&IV}, we have

%\begin{proposition}\label{prop_H(X)singularloci}
%    Let $X$ be a smooth cubic hypersurface with $\dim (X)\ge 3$. Then
% the singularities of $\tilde{D}_F'$ is parameterized by the triples
%$$(p,L,v),$$
%where $L$ is a higher triple line $L$, $p\in L$ is a point, and $v\in H^{0}()$ is plane uniquely associated with $L$ (cf. Remark \ref{remark_pv}) and satisfies $P^2\cap X=3L$.
%\end{proposition}
\begin{proposition}
When $\dim(X)\ge 4$, each singularity of $H(X)$ corresponds to a triple $$(p,L,P^3)$$ 
where $L$ is a higher triple line $L$, $p\in L$ is a point, and $P^3$ is a linear 3-dimensional subspace of $P^n$ containing $L$ and a normal direction $v\in H^0(\mathcal{O}_L(1))$, with $\mathcal{O}_L(1)\subseteq N_{L|X}$. This data determines a type (IV) subscheme $Z_{IV}$ supported on $L$, and an embedded point supported on $p$ and contained in $P^3$.
\end{proposition}
% a  supported on a point $p\in L$ and determine a normal direction to $P^2$.
Here, the 3-plane $P^3$ determines the normal direction of the embedded point to $P^2$ and vice versa. More generally, there is a morphism
$$\pi: H(X)\to Gr(4,n+1)$$
by assigning each $Z\in H(X)$ to the unique 3-plane $\pi(Z)\cong \mathbb P^3$ containing $Z$. 

One may ask the following question.
\begin{question}\label{question_piZ}
  Suppose $Z_{IV}\in H(X)$ is a singular point. How to describe the 3-plane $\pi(Z_{IV})$ containing $Z_{IV}$?
\end{question}

If $Z_{IV}\in H(X)$ is a singularity, then the 3-plane $\pi(Z_{IV})$. Suppose it is not contained in $X$ (this is the case when $\dim(X)\le 4$), then the intersection $\pi(Z_{IV})\cap X$ is a cubic surface, which also contains $Z_{IV}$ as a closed subscheme. We denote it by $C$. We observe that
\begin{proposition}\label{prop_singcubicsurface}
    The cubic surface $C$ is either 
    \begin{itemize}
        \item  a cone of a planer cuspidal cubic curve, or 
        \item the union of a plane and a quadric cone meeting tangentially along $L$.
    \end{itemize}
\end{proposition}
\begin{proof}
  Let $P^2$ be the unique plane as in Remark \ref{remark_pv}, and $P^{n-2}$ be the $(n-2)$ plane tangent to $X$ along $L$ (cf. Lemma \ref{lemma_(n-3)plane}). Then, any linear subspace $P^k$ such that
$$P^2\subseteq P^k\subseteq P^{n-2}$$
  satisfies that $P^k\cap X$ is singular along the line $L$ and has transversal $A_2$ singularities along $L$. In particular, take $P^k=\pi(Z)$, then $\pi(Z)\cap X$ has equation 
  $$l(x_0,x_1)x_2^2+c(x_2,x_3)=0,$$
  where $l$ is a linear form and $c$ is a cubic form. Then if $x_2\nmid c(x_2,x_3)$, the affine curve $x_2^2+c(x_2,x_3)=0$ has a cusp at $(0,0)$ and $C$ is a cone over it. Otherwise, the curve is the union of a line and conic tangent at a point, so its cone $C$ is the union of a plane and quadric cone meeting along a line.
\end{proof}

In the first case, there is a unique $P^2$ such that $P^2\cap X=3L$ (cf. Remark \ref{remark_pv}), while in the second case, $P^2$ is contained in $X$. In either case, $L$ is a triple line by Definition \ref{def_tripleline}.

\begin{remark}\normalfont\label{remark_hierarchy}
    There is a hierarchy of cubic surfaces based on singularities and codimension of parameter space (cf. \cite[p.13]{LLSvS}, \cite[App.]{GG}), which is roughly $$\textup{normal w/} 
    \begin{cases}
        ADE\  \textup{singualarities}\\
        \textup{elliptic singularity}
    \end{cases}
    \rightsquigarrow \textup{non-normal, integral}\begin{cases}
        X_6\\
        X_7\\
        X_8\\
        X_9
    \end{cases}\rightsquigarrow \textup{non-integral cubic surfaces}.$$
    Here we use the notations in \cite{LLSvS}, and the subscribe denotes the codimension of the parameter space. The cone of the cuspidal curve is $X_9$.
\end{remark}

\bibliographystyle{alpha}
\bibliography{bibfile}

\end{document}